\documentclass[a4paper,11pt]{article}
\usepackage[OT1,T1]{fontenc}
\usepackage{a4wide}
\usepackage{amsmath,amssymb}
\usepackage{mathrsfs}
\usepackage[all]{xy}
\SelectTips{cm}{10}
\usepackage{tikz}
\usetikzlibrary{calc}
\usepackage{ctable,array}
\usepackage[colorlinks=false,pdfpagemode=None,pdfstartview=FitH]{hyperref}

\newtheorem{theorem}{Theorem}[section]
\newtheorem{lemma}[theorem]{Lemma}
\newtheorem{proposition}[theorem]{Proposition}
\newtheorem{corollary}[theorem]{Corollary}
\newenvironment{proof}{\trivlist
  \item[\hskip\labelsep{\itshape Proof.}]\upshape}{\nobreak\noindent
  $\square$\endtrivlist}
\newenvironment{other}[1]{\refstepcounter{theorem}\trivlist
  \item[\hskip\labelsep{\itshape #1~\thesection.\arabic{theorem}.}]
  \upshape}{\endtrivlist}

\DeclareMathOperator\Hom{Hom}
\DeclareMathOperator\Ext{Ext}
\DeclareMathOperator\im{im}
\DeclareMathOperator\wt{wt}
\newcommand\GL{{\mathbf{GL}}}
\newcommand\lex{\mathrm{lex}}
\newcommand\Pol{\mathrm{Pol}}
\newcommand\pol{\mathrm{pol}}
\newcommand\str{\mathrm{str}}
\newcommand\bsm{\begin{smallmatrix}}
\newcommand\esm{\end{smallmatrix}}

\begin{document}
\title{The canonical basis and the quantum Frobenius morphism}
\author{Pierre Baumann}
\date{}
\maketitle

\begin{abstract}
\noindent
The first goal of this paper is to study the amount of compatibility
between two important constructions in the theory of quantized
enveloping algebras, namely the canonical basis and the quantum
Frobenius morphism. The second goal is to study orders with which the
Kashiwara crystal $B(\infty)$ of a symmetrizable Kac-Moody algebra
can be endowed; these orders are defined so that the transition
matrices between bases naturally indexed by $B(\infty)$ are lower
triangular.
\end{abstract}

{\small Mathematics Subject Classification: 17B10 (Primary)\thinspace;
05E10 (Secondary).\\
Keywords: crystal basis, canonical basis, quantum Frobenius morphism.}

\section{Introduction}
Let $\mathfrak g=\mathfrak n_-\oplus\mathfrak h\oplus\mathfrak n_+$
be the triangular decomposition of a symmetrizable Kac-Moody algebra.
Twenty years ago, with the help of the quantized enveloping algebra
$U_q(\mathfrak g)$, Lusztig and Kashiwara constructed a basis
$\mathbf B$ in the enveloping algebra $U(\mathfrak n_+)$, called the
canonical basis, whose properties make it particularly well suited
to the study of integrable highest weight $\mathfrak g$-modules
\cite{Lusztig90,Lusztig93,Kashiwara91}. Subsequently, Kashiwara
studied the combinatorics of $\mathbf B$ with his abstract notion
of crystal~\cite{Kashiwara95}, while Lusztig, during his investigation
of the geometric problems raised by his construction of $\mathbf B$,
was eventually led to the definition of a second basis, called the
semicanonical basis~\cite{Lusztig00}.

The quantized enveloping algebra setting leads to other useful tools,
as the quantum Frobenius morphism $Fr$ and its splitting $Fr'$. Using
these maps, Kumar and Littelmann algebraized the proofs that use Mehta
and Ramanathan's Frobenius splitting for algebraic groups in positive
characteristic~\cite{KumarLittelmann02}. Littlemann also used the
quantum Frobenius splitting to define a basis in all simple
$\mathfrak g$-modules from the combinatorics of LS-paths, completing
thereby the program of standard monomial theory~\cite{Littelmann98}.

It is thus desirable to study the extent of compatibility between
these two constructions, the canonical basis and the quantum
Frobenius map. The best behavior would be that $Fr$ and $Fr'$
map a basis vector to a basis vector or to zero, in a way that
admits a combinatorial characterization.

Lusztig observed that $Fr$ commutes with the comultiplication.
It is then tempting to use duality to understand the situation. More
precisely, the graded dual of $U(\mathfrak n_+)$ can be identified
with the algebra $\mathbb Q[N]$ of regular functions on $N$, the
unipotent group with Lie algebra $\mathfrak n_+$, and the dual of the
canonical basis behaves rather nicely with respect to the multiplication
of $\mathbb Q[N]$~\cite{BerensteinZelevinsky93}. Unfortunately,
rather nicely does not mean perfect agreement, as was shown by
Leclerc~\cite{Leclerc03}. We will see in Section~\ref{ss:FrobCexA5D4}
that $Fr$ fails to be fully compatible with the canonical basis at
the same spot where Leclerc found his counterexamples. This failure
partially answers a question raised by McGerty (\cite{McGerty10},
Remark~5.10), asking whether his construction of $Fr$ in the context
of Hall algebras can be lifted to the level of perverse sheaves.

In view of the applications, the study of $Fr'$ is perhaps even more
important. An encouraging fact is that $Fr'$ is compatible with
$\mathbf B$ in small rank (type $A_1$, $A_2$, $A_3$ and $B_2$).
Alas, in general, $Fr'$ is not compatible with $\mathbf B$
(see Section~\ref{ss:FrobCexA5D4}).

One can however obtain a form of compatibility between $Fr$, $Fr'$
and $\mathbf B$ by focusing on leading terms, that is, by neglecting
terms that are smaller. This result is hardly more than an
observation, but it invites us to study the orders with which
$\mathbf B$ can be endowed. Since as a set $\mathbf B$ is just
Kashiwara's cristal $B(-\infty)$, we will in fact investigate
orders on $B(-\infty)$.

We will present two natural ways to order $B(-\infty)$. In the first
method, one checks the values of the functions $\varepsilon_i$ and
$\varphi_i$; after stabilization by the crystal operations $\tilde e_i$
and $\tilde f_i$, one obtains an order $\leq_{\str}$ (Section
\ref{ss:StrOrder}).

The second method, which works only when $\mathfrak g$ is finite
dimensional, relies on the notion of MV polytope~\cite{Anderson03,
Kamnitzer07,Kamnitzer10}: to each $b\in B(-\infty)$ is associated a
convex polytope $\Pol(b)\subseteq\mathfrak h^*$, and the containment
of these polytopes defines an order $\leq_{\pol}$ on $B(-\infty)$
(Section~\ref{ss:PolOrder}).

The plan of this paper is as follows. In Section~\ref{se:CrysOp},
we review the properties of the canonical basis that are important
from a combinatorial viewpoint, following the methods set up by
Kashiwara, Berenstein, Zelevinsky, and their coauthors. This leads
us quite naturally to the definition of the order $\leq_{\str}$.
In Section~\ref{se:IncMVPol}, we recall the definition of the MV
polytope of an element of $B(-\infty)$ and explain how to numerically
test whether two elements of $B(-\infty)$ are comparable w.r.t.\
the order $\leq_{\pol}$. In Section~\ref{se:FurEx}, we assume that
the Cartan matrix of $\mathfrak g$ is symmetric and recall
Lusztig's construction of the canonical and semicanonical bases;
we relate $\leq_{\pol}$ to the degeneracy order between quiver
representations (Proposition~\ref{pr:PolOrderDeg}) and we show that
the transition matrix between the canonical basis and the
semicanonical basis is lower unitriangular, for both $\leq_{\str}$
and $\leq_{\pol}$ (Theorem~\ref{th:CompCanSemican}). In
Section~\ref{se:StudA5D4}, we revisit Leclerc's counterexamples in
type $A_5$ and $D_4$; Theorem~\ref{th:MainA5D4} provides the
expansion on $\mathbf B$ of certain monomials in the Chevalley
generators of $U(\mathfrak n_+)$. This result is used in
Section~\ref{se:CanFrob} to show that the Frobenius morphism $Fr$
and its splitting $Fr'$ are not fully compatible with $\mathbf B$.

The author wishes to thank P.~Caldero, J.~Kamnitzer, B.~Leclerc,
P.~Littelmann and C.~Sabbah for many discussions connected to the
research reported here. He also thanks A.~Kleshchev for the reference
to Berenstein and Kazhdan's work~\cite{BerensteinKazhdan07}. Results
presented in Sections~\ref{ss:StatRes} and~\ref{ss:FrobCexA5D4} were
found with the help of a computer running GAP and its package
QuaGroup~\cite{GAP4,DeGraaf07}.

The material of Section~\ref{ss:FrobCexA5D4} was presented at the
conference \href{http://www.math.unipd.it/~carnoval/olivia.html}
{\textit{Una giornata di Algebra a Roma}} dedicated to the memory
of Olivia Rossi-Doria. Olivia died in June 2006 at the age of 35
and is missed by all her friends.

\section{Crystal operations}
\label{se:CrysOp}
For all this paper, we fix a symmetrizable generalized Cartan matrix
$A=(a_{i,j})$, with rows and columns indexed by a finite set $I$. We
choose a $\mathbb Q$-vector space $\mathfrak h$, with a basis indexed
by~$I$. We denote by $(\alpha_i)_{i\in I}$ the dual basis in
$\mathfrak h^*$ and we define elements $\alpha_i^\vee$ in $\mathfrak h$
by the equation $\langle\alpha_i^\vee,\alpha_j\rangle=a_{i,j}$.
We fix a lattice $P\subseteq\mathfrak h^*$ such that
$$\{\alpha_i\mid i\in I\}\subseteq P\subseteq\{\lambda\in\mathfrak h^*\mid
\forall i\in I,\ \langle\alpha_i^\vee,\lambda\rangle\in\mathbb Z\}.$$
We denote the $\mathbb N$-span of the simple roots $\alpha_i$ by $Q_+$.

\subsection{Crystals}
\label{ss:Crystals}
A crystal in the sense of Kashiwara \cite{Kashiwara95} is a set $B$
endowed with maps
$$\wt:B\to P,\quad\varepsilon_i,\varphi_i:B\to\mathbb Z
\quad\text{and}\quad\tilde e_i,\tilde f_i:B\to B\sqcup\{0\},$$
for each $i\in I$. The element $0$ here is a ghost element added to
$B$ so that $\tilde e_i$ and $\tilde f_i$ are everywhere defined.
One requires that $\langle\alpha_i^\vee,\wt(b)\rangle=
\varphi_i(b)-\varepsilon_i(b)$ for each $b\in B$. The operators
$\tilde e_i$ and $\tilde f_i$ are mutually converse partial bijections:
$b''=\tilde e_ib'$ if and only if $\tilde f_ib''=b'$, and when these
equalities hold,
$$\wt(b'')=\wt(b')+\alpha_i,\quad\varepsilon_i(b'')=\varepsilon_i(b')-1
\quad\text{and}\quad\varphi_i(b'')=\varphi_i(b')+1.$$

We say that a crystal $B$ is lower normal if for each $b\in B$,
the number $p=\varphi_i(b)$ is the largest integer $p\in\mathbb N$
such that $\tilde f_i^pb$ is defined. (In other words,
$\tilde f_i^pb\in B$ and $\tilde f_i^{p+1}b=0$.)

Let $B$ be a lower normal crystal. For $i\in I$ and $b\in B$, we set
$\tilde f_i^{\max}b=\tilde f_i^{\varphi_i(b)}b$. In addition, for a
finite sequence $\mathbf i=(i_1,\ldots,i_\ell)$ of elements in $I$,
we define maps $\Phi_{\mathbf i}:B\to\mathbb N^\ell$ and
$\widetilde F_{\mathbf i}:B\to B$ in the following fashion. Given
$b\in B$, we put $b_0=b$, and for $k\in\{1,\ldots,\ell\}$, we set
$n_k=\varphi_{i_k}(b_{k-1})$ and $b_k=\tilde f_{i_k}^{\max}b_{k-1}$.
With these notations,
$$\Phi_{\mathbf i}(b)=(n_1,\ldots,n_\ell)\quad\text{and}\quad
\widetilde F_{\mathbf i}(b)=b_\ell.$$
Of course, the datum of $\Phi_{\mathbf i}(b)$ and of
$\widetilde F_{\mathbf i}(b)$ fully determines $b$.
The map $\Phi_{\mathbf i}$ is usually called the string
parametrization in direction $\mathbf i$
\cite{BerensteinZelevinsky93,BerensteinZelevinsky01,Kashiwara93b}.

\subsection{Bases of canonical type}
\label{ss:BasesOCT}
Let $\mathbf f$ be the $\mathbb Q$-algebra generated by elements
$\theta_i$, for $i\in I$, submitted to the relations
\begin{equation}
\label{eq:DefAlgF}
\sum_{p+q=1-a_{i,j}}(-1)^p\frac{\theta_i^p}{p!}\theta_j
\frac{\theta_i^q}{q!}=0
\end{equation}
for all $i\neq j$ in $I$. This algebra $\mathbf f$ is naturally graded
by $Q_+$ (gradation by the weight), the generator $\theta_i$ being of
weight $\alpha_i$; we write $\mathbf f=\bigoplus_{\nu\in Q_+}\mathbf
f_\nu$. In addition, $\mathbf f$ has an antiautomorphism $\sigma$
which fixes the~$\theta_i$. As usual, we introduce the divided powers
$\theta_i^{(n)}=\theta_i^n/n!$.

Let $\mathfrak g=\mathfrak n_-\oplus\widehat{\mathfrak h}\oplus
\mathfrak n_+$ be the Kac-Moody algebra defined by the Cartan
matrix $A$. By the Gabber-Kac Theorem~\cite{GabberKac81}, there
are isomorphisms $x\mapsto x^\pm$ from $\mathbf f$ onto
$U(\mathfrak n^\pm)$, that map the generators $\theta_i$ to the
elements $e_i$ or $f_i$ of $\mathfrak g$.

A basis $\mathbf B$ of $\mathbf f$ is said to be of canonical type
if it satisfies the conditions~\ref{it:BOCTa}--\ref{it:BOCTe} below:
\begin{enumerate}
\item
\label{it:BOCTa}
The elements of $\mathbf B$ are weight vectors.
\item
\label{it:BOCTb}
$1\in\mathbf B$.
\item
\label{it:BOCTc}
Each right ideal $\theta_i^p\mathbf f$ is spanned by a subset of
$\mathbf B$.
\item
\label{it:BOCTd}
In the bases induced by $\mathbf B$, the left multiplication by
$\theta_i^{(p)}$ from $\mathbf f/\theta_i\mathbf f$ onto
$\theta_i^p\mathbf f/\theta_i^{p+1}\mathbf f$ is given by a
permutation matrix.
\item
\label{it:BOCTe}
$\mathbf B$ is stable by $\sigma$.
\end{enumerate}

Let $N$ be the unipotent group with Lie algebra $\mathfrak n_+$.
The graded dual of $\mathbf f$ can be identified with the algebra
$\mathbb Q[N]$ of regular functions on~$N$. Berenstein, Zelevinsky,
and their coauthors extensively studied bases of $\mathbb Q[N]$ that
enjoy axioms dual to the conditions \ref{it:BOCTa}--\ref{it:BOCTd}
above. More precisely, the bases considered in
\cite{GelfandZelevinsky85,RetakhZelevinsky88} (the so-called
good bases) forget about condition \ref{it:BOCTd}, the string
bases of~\cite{BerensteinZelevinsky93} satisfy stronger axioms
inspired by the positivity properties of Lusztig's canonical basis in the
quantized enveloping algebra $U_q(\mathfrak n_+)$, and the perfect
bases of~\cite{BerensteinKazhdan07} relax the normalization constraint
in condition~\ref{it:BOCTd}. Our definition above is of course
directly inspired from these works.

Adopting this dual setting allows to exploit the multiplicative
structure of the function algebra $\mathbb Q[N]$. We will however
stick with $\mathbf f$, because the quantum Frobenius splitting $Fr'$
is more easily defined on the enveloping algebra side.

\subsection{Examples}
\label{ss:ExamplesBCT}
Since $A$ is symmetrizable, the half-quantum group $U^-_q(\mathfrak g)$
has a canonical basis \cite{Lusztig91,Lusztig93,Kashiwara91}. After
specialization at $q=1$, one gets a basis of $\mathbf f$ of canonical
type, by Theorems~14.3.2 and~14.4.3 in~\cite{Lusztig93}, or by Theorem~7
in \cite{Kashiwara91}, Proposition~5.3.1 in~\cite{Kashiwara93a} and
Theorem~2.1.1 in~\cite{Kashiwara93b}.

If $A$ is symmetric, then $\mathbf f$ can be endowed with the
semicanonical basis~\cite{Lusztig00}. Again, this is a basis of
canonical type, by Theorems~3.1 and~3.8 and by the proof of
Lemma~2.5 in~\cite{Lusztig00}.

If $A$ is of finite type, then one can use the geometric Satake
correspondence \cite{Ginzburg95} and Mirkovi\'c-Vilonen cycles
\cite{MirkovicVilonen04} to define a basis of $\mathbf f$. It can be
shown that this basis is of canonical type.

One major interest of bases of canonical type is that they induce bases
in irreducible integrable highest weight representations of $\mathfrak g$.
In more details, let $M$ be an irreducible integrable highest weight
representation of $\mathfrak g$, let $m$ be a highest weight vector of
$V$, and let $\lambda$ be the weight of $m$. Then the map
$a:x\mapsto x^-m$ from $\mathbf f$ to $M$ is surjective with kernel
$\sum_{i\in I}\mathbf f\theta_i^{n_i}$, where
$n_i=1+\langle\alpha_i^\vee,\lambda\rangle$ (see \cite{Kac90},
Corollary~10.4). If $\mathbf B$ is a basis of canonical type of
$\mathbf f$, then $\ker a$ is spanned by a subset of $\mathbf B$,
hence $a(\mathbf B)\setminus\{0\}$ is a basis of $\im a=M$. Moreover,
the analysis done in \cite{Kashiwara93b}, Section~3.2 shows that
this basis is compatible with all Demazure submodules of $M$.
In particular, the semicanonical basis is compatible with
Demazure submodules, a result first obtained by Savage
(\cite{Savage06}, Theorem~7.1).

\subsection{The crystal $B(-\infty)$}
\label{ss:CrysBInfinity}
Each basis $\mathbf B$ of canonical type of $\mathbf f$ is endowed
with the structure of a lower normal crystal, as follows. The
application $\wt$ maps a vector $b\in\mathbf B$ to its weight.
Let $i\in I$ and let $b\in\mathbf B$. We define $\varphi_i(b)$ as the
largest $p\in\mathbb N$ such that $b\in\theta_i^p\mathbf f$ and we set
$\varepsilon_i(b)=\varphi_i(b)-\langle\alpha_i^\vee,\wt(b)\rangle$.
Thus, given $i\in I$ and $p\in\mathbb N$, the images
of the elements $\{b\in\mathbf B\mid\varphi_i(b)=p\}$ form a basis
of $\theta_i^p\mathbf f/\theta_i^{p+1}\mathbf f$.
The partial bijections $\tilde e_i$ and $\tilde f_i$ are set up so
that when $\varphi_i(b)=p$, the element $\tilde f_i^{\max}b$ is
the element of $\mathbf B$ such that
$\theta_i^{(p)}\,\tilde f_i^{\max}b\equiv b$ modulo
$\theta_i^{p+1}\mathbf f$.

It turns out that any two bases of canonical type have the same
underlying crystal. This fact is a particular case of a result by
Berenstein and Kazhdan~\cite{BerensteinKazhdan07}. For the convenience
of the reader, we now recall the proof.

Given a basis $\mathbf B$ of a vector space $V$, we denote by $b^*$
the element of the dual basis that corresponds to $b\in\mathbf B$; in
other words, $x=\sum_{b\in\mathbf B}\langle b^*,x\rangle\;b$ for
all vectors $x\in V$.

\begin{lemma}
\label{le:UniqBij}
Let $V$ be a finite dimensional vector space and let $\mathbf B'$ and
$\mathbf B''$ be two bases of $V$. For $a\in\{1,2\}$, let $\mathcal C_a$
be a partially ordered set and let $b'_a:\mathcal C_a\to\mathbf B'$ and
$b''_a:\mathcal C_a\to\mathbf B''$ be bijections; thus both the rows
and the columns of the transition matrix between $\mathbf B'$ and
$\mathbf B''$ are indexed by $\mathcal C_a$. If the transition matrix
between $\mathbf B'$ and $\mathbf B''$ is lower unitriangular with
respect to both indexings $(\mathcal C_1,b'_1,b''_1)$ and
$(\mathcal C_2,b'_2,b''_2)$, then
$b''_1\circ(b'_1)^{-1}=b''_2\circ(b'_2)^{-1}$.
\end{lemma}
\begin{proof}
Let $\Sigma=(b''_1)^{-1}\circ b''_2\circ(b'_2)^{-1}\circ b'_1$.
Suppose that $\Sigma$ is not the identity permutation of
$\mathcal C_1$. Let $m\in\mathcal C_1$ be a maximal element in a
nontrivial cycle of $\Sigma$ and let $n=(b'_2)^{-1}\circ b'_1(m)$.
Then $\bigl\langle(b''_1(\Sigma(m)))^*,b'_1(m)\bigr\rangle=
\bigl\langle(b''_2(n))^*,b'_2(n)\bigr\rangle=1$, and therefore
$\Sigma(m)\geq m$ in $\mathcal C_1$, by unitriangularity of the
transition matrix. This contradicts the choice of $m$. We conclude
that $\Sigma$ is the identity.
\end{proof}

\begin{lemma}
\label{le:HighWt}
Let $\mathbf B$ be a basis of canonical type and let $b\neq1$ be
an element of $\mathbf B$. Then there exists $i\in I$ such that
$\varphi_i(b)>0$.
\end{lemma}
\begin{proof}
For each $i\in I$, the right ideal $\theta_i\mathbf f$ is a
spanned by a subset of $\mathbf B$, namely by
$\mathbf B_{i;>0}=\mathbf B\cap\theta_i\mathbf f=\{b\in\mathbf
B\mid\varphi_i(b)>0\}$. The subspace $\sum_{i\in I}\theta_i\mathbf f$
is thus spanned by $\bigcup_{i\in I}\mathbf B_{i;>0}$. Any element of
$\mathbf B$ in this subspace therefore belongs to
$\bigcup_{i\in I}\mathbf B_{i;>0}$.
\end{proof}

Given an integer $\ell\geq0$, we endow $\mathbb N^\ell$
with the lexicographic order $\leq_{\lex}$.

\begin{proposition}
\label{pr:LexOrd}
Let $\mathbf B'$ and $\mathbf B''$ be two bases of canonical type of
$\mathbf f$. Let $\mathbf i=(i_1,\ldots,i_\ell)$ be a finite sequence
of elements of $I$.
\begin{enumerate}
\item
\label{it:LOa}
Let $(b',b'')\in\mathbf B'\times\mathbf B''$. If
$\langle(b'')^*,b'\rangle\neq0$, then
$\Phi_{\mathbf i}(b')\leq_{\lex}\Phi_{\mathbf i}(b'')$.
\item
\label{it:LOb}
In~\ref{it:LOa}, if moreover $\Phi_{\mathbf i}(b')=\Phi_{\mathbf
i}(b'')$, then $\langle(b'')^*,b'\rangle=\langle(\widetilde
F_{\mathbf i}b'')^*,\widetilde F_{\mathbf i}b'\rangle$.
\item
\label{it:LOc}
For each $b''\in\mathbf B''$, there is $b'\in\mathbf B'$
such that $\langle(b'')^*,b'\rangle\neq0$ and
$\Phi_{\mathbf i}(b')=\Phi_{\mathbf i}(b'')$.
\end{enumerate}
\end{proposition}
\begin{proof}
The proof proceeds by induction on the length $\ell$ of $\mathbf i$.
The result is trivial if $\ell=0$. We now assume that $\ell>0$ and
that the result holds for the sequence $\mathbf j=(i_2,\ldots,i_\ell)$.

Let $b'\in\mathbf B'$, let $n_1=\varphi_{i_1}(b')$, and let
$b'_1=\tilde f_{i_1}^{\max}b'$. Let us write
$$b'_1=\sum_{b''_1\in\mathbf B''}\langle(b''_1)^*,b'_1\rangle\;b''_1
\equiv\sum_{\substack{b''_1\in\mathbf B''\\\varphi_{i_1}(b''_1)=0}}\langle
(b''_1)^*,b'_1\rangle\;b''_1\pmod{\theta_{i_1}\mathbf f}.$$
Multiplying on the left by $\theta_{i_1}^{(n_1)}$, we obtain
$$b'\equiv\theta_{i_1}^{(n_1)}b'_1
\equiv\sum_{\substack{b''_1\in\mathbf B''\\\varphi_{i_1}(b''_1)=0}}\langle
(b''_1)^*,b'_1\rangle\;\theta_{i_1}^{(n_1)}b''_1\\
\equiv\sum_{\substack{b''_1\in\mathbf B''\\\varphi_{i_1}(b''_1)=0}}\langle
(b''_1)^*,b'_1\rangle\;\tilde e_{i_1}^{n_1}b''_1
\pmod{\theta_{i_1}^{n_1+1}\mathbf f}.$$
Since $\theta_{i_1}^{n_1+1}\mathbf f$ is spanned by
$\{b''\in\mathbf B''\mid\varphi_{i_1}(b'')>n_1\}$, we see that
if $\langle(b'')^*,b'\rangle\neq0$, then either
$\varphi_{i_1}(b'')>n_1$, or $b''=\tilde e_{i_1}^{n_1}b''_1$ with
$\varphi_{i_1}(b''_1)=0$ and $\langle(b''_1)^*,b'_1\rangle=\langle
(b'')^*,b'\rangle$. Assertions~\ref{it:LOa} and~\ref{it:LOb} for
$b'$ and $\mathbf i$ thus readily follow from the corresponding
assertions for $b'_1$ and $\mathbf j$.

Now let $b''\in\mathbf B''$, let $n_1=\varphi_{i_1}(b'')$, and let
$b''_1=\tilde f_{i_1}^{\max}b''$. By induction, there exists
$b'_1\in\mathbf B'$ such that $\langle(b''_1)^*,b'_1\rangle\neq0$
and $\Phi_{\mathbf j}(b'_1)=\Phi_{\mathbf j}(b''_1)$. We then have
$0\leq\varphi_{i_1}(b'_1)\leq\varphi_{i_1}(b''_1)=0$. Let
$b'=\tilde e_{i_1}^{n_1}b'_1$. Then
$\Phi_{\mathbf i}(b')=\Phi_{\mathbf i}(b'')$, and also, by
the reasoning used in the proof of~\ref{it:LOa},
$\langle(b'')^*,b'\rangle=\langle(b''_1)^*,b'_1\rangle\neq0$.
This shows~\ref{it:LOc}.
\end{proof}

\begin{theorem}
\label{th:BerKazh}
Let $\mathbf B'$ and $\mathbf B''$ be two bases of canonical type of
$\mathbf f$. Then there is a unique bijection
$\Xi:\mathbf B'\to\mathbf B''$ such that
$\Phi_{\mathbf i}(b')=\Phi_{\mathbf i}(\Xi(b'))$ for any finite
sequence $\mathbf i$ of elements of $I$ and any $b'\in\mathbf B'$.
Moreover, $\Xi$ is an isomorphism of crystals which commutes with
the action of the involution $\sigma$.
\end{theorem}
\begin{proof}
Proposition~\ref{pr:LexOrd} generalizes in an obvious way to infinite
sequences $\mathbf i=(i_1,i_2,\ldots)$ of elements of $I$, because
for any element $b$ in a basis of canonical type, the sequence
$$(b,\tilde f_{i_1}^{\max}b,\tilde f_{i_2}^{\max}\tilde
f_{i_1}^{\max}b,\ldots)$$
eventually becomes constant for weight reasons. In this situation,
if each $i\in I$ appears an infinite number of times in $\mathbf i$,
then the limit of this sequence is necessarily equal to $1$, by
Lemma~\ref{le:HighWt}.

Let us fix such a sequence $\mathbf i$. Then for any basis $\mathbf B$
of canonical type, $\Phi_{\mathbf i}$ is an injective map from
$\mathbf B$ to the set $\mathbb N^{(\infty)}$ of sequences of
non-negative integers with finitely many nonzero terms.
In addition, $\mathcal C_{\mathbf i}=\Phi_{\mathbf i}(\mathbf B)$
does not depend on $\mathbf B$, by Proposition~\ref{pr:LexOrd}~\ref{it:LOc}.

We can thus index any basis $\mathbf B$ of canonical type by
$\mathcal C_{\mathbf i}$. Using this for two bases $\mathbf B'$
and $\mathbf B''$ of canonical type, we moreover deduce from
Proposition~\ref{pr:LexOrd}~\ref{it:LOa} and~\ref{it:LOb} that
the transition matrix is lower unitriangular if we endow
$\mathcal C_{\mathbf i}$ with the lexicographic order on
$\mathbb N^{(\infty)}$. From Lemma~\ref{le:UniqBij}, we conclude
that the bijection $\Xi:\mathbf B'\to\mathbf B''$ defined by the
diagram
$$\xymatrix@R=6pt{\mathbf B'\ar[dr]_{\Phi_{\mathbf i}}\ar[rr]^\Xi&&
\mathbf B''\ar[dl]^{\Phi_{\mathbf i}}\\&\mathcal C_{\mathbf i}&}$$
does not depend on $\mathbf i$. (Lemma~\ref{le:UniqBij} can be
applied in our context because the weight spaces of $\mathbf f$
are finite dimensional.)

Lastly, we observe that the transition matrix is also lower
unitriangular if we use the indexations
$\sigma\circ\Phi_{\mathbf i}^{-1}$ from $\mathcal C_{\mathbf i}$
to $\mathbf B'$ and $\mathbf B''$, and so
$\Xi=\sigma\circ\Xi\circ\sigma$, again by Lemma~\ref{le:UniqBij}.
\end{proof}

The crystal common to all bases of canonical type is denoted by
$B(-\infty)$.

\begin{other}{Remark}
\label{rk:StringCone}
Let $\mathbf i$ and $\mathcal C_{\mathbf i}$ be as in the proof of
Theorem~\ref{th:BerKazh}. To each sequence $\mathbf n=(n_1,n_2,\ldots)$
in $\mathcal C_{\mathbf i}$, define $\Theta_{\mathbf i}^{(\mathbf n)}=
\theta_{i_1}^{(n_1)}\theta_{i_2}^{(n_2)}\cdots$. Let $\mathbf B$
be a basis of canonical type of $\mathbf f$. Arguing as in the
proof of Proposition~\ref{pr:LexOrd}, one shows that
$$\bigl\langle b^*,\Theta_{\mathbf i}^{(\mathbf n)}\bigr\rangle=
\begin{cases}
1&\text{if $\mathbf n=\Phi_{\mathbf i}(b)$,}\\
0&\text{if $\mathbf n\not\leq_{\lex}\Phi_{\mathbf i}(b)$.}
\end{cases}$$
It follows that the elements $\Theta_{\mathbf i}^{(\mathbf n)}$ form
a basis of $\mathbf f$. This result is due to Lakshmibai
(\cite{Lakshmibai95}, Theorems~6.5 and~6.6).
\end{other}

\subsection{The string order on $B(-\infty)$}
\label{ss:StrOrder}
Given $(b',b'')$ and $(c',c'')$ in $B(-\infty)^2$, we write
$(b',b'')\approx(c',c'')$ if one of the following three conditions
holds:
\begin{itemize}
\item
There is $i\in I$ such that $\varphi_i(b')=\varphi_i(b'')$ and
$(c',c'')=(\tilde e_ib',\tilde e_ib'')$.
\item
There is $i\in I$ such that $\varphi_i(b')=\varphi_i(b'')>0$ and
$(c',c'')=\bigl(\tilde f_ib',\tilde f_ib''\bigr)$.
\item
$(c',c'')=(\sigma(b'),\sigma(b''))$.
\end{itemize}

Given $(b',b'')\in B(-\infty)^2$, we write $b'\leq_{\str}b''$
if $b'$ and $b''$ have the same weight and if for any finite
sequence of elementary moves
$$(b',b'')=(b'_0,b''_0)\approx(b'_1,b''_1)\approx\cdots
\approx(b'_\ell,b''_\ell),$$
one has $\varphi_i(b'_\ell)\leq\varphi_i(b''_\ell)$ for all $i\in I$.

\begin{proposition}
\label{pr:StrOrder}
\begin{enumerate}
\item
\label{it:SOa}
The relation $\leq_{\str}$ is an order on $B(-\infty)$.
\item
\label{it:SOb}
The transition matrix between two bases of canonical type is lower
unitriangular w.r.t.\ the order $\leq_{\str}$.
\end{enumerate}
\end{proposition}
\begin{proof}
The relation $\leq_{\str}$ is obviously reflexive. Let us
show that it is transitive. Suppose that $b'\leq_{\str}b''$ and
$b''\leq_{\str}b'''$ and let us consider a finite sequence
$$(b',b''')=(b'_0,b'''_0)\approx(b'_1,b'''_1)\approx\cdots
\approx(b'_\ell,b'''_\ell).$$
By induction, we construct elements $b''_0=b''$, $b''_1$, \dots,
$b''_\ell$ such that
$$(b',b'')=(b'_0,b''_0)\approx(b'_1,b''_1)\approx\cdots
\approx(b'_\ell,b''_\ell)$$
and
$$(b'',b''')=(b''_0,b'''_0)\approx(b''_1,b'''_1)\approx\cdots
\approx(b''_\ell,b'''_\ell).$$
More precisely, assuming that $b''_{k-1}$ is constructed, the
assumption $b'\leq_{\str}b''$ and $b''\leq_{\str}b'''$ implies that
$\varphi_i(b'_{k-1})\leq\varphi_i(b''_{k-1})\leq\varphi_i(b'''_{k-1})$
for all $i\in I$. Now:
\begin{itemize}
\item
If $\varphi_i(b'_{k-1})=\varphi_i(b'''_{k-1})$ and
$(b'_k,b'''_k)=(\tilde e_ib'_{k-1},\tilde e_ib'''_{k-1})$,
then we note that
$\varphi_i(b'_{k-1})=\varphi_i(b''_{k-1})=\varphi_i(b'''_{k-1})$,
and we set $b''_k=\tilde e_ib''_{k-1}$.
\item
If $\varphi_i(b'_{k-1})=\varphi_i(b'''_{k-1})>0$ and
$(b'_k,b'''_k)=\bigl(\tilde f_ib'_{k-1},\tilde f_ib'''_{k-1}\bigr)$,
then we note that
$\varphi_i(b'_{k-1})=\varphi_i(b''_{k-1})=\varphi_i(b'''_{k-1})>0$
and we set $b''_k=\tilde f_ib''_{k-1}$.
\item
If $(b'_k,b'''_k)=(\sigma(b'_{k-1}),\sigma(b'''_{k-1}))$, then we
simply set $b''_k=\sigma(b''_{k-1})$.
\end{itemize}

Lastly, we note that if $b'\leq_{\str}b''$, then
$\Phi_{\mathbf i}(b')\leq_{\lex}\Phi_{\mathbf i}(b'')$ for any
sequence $\mathbf i$ of elements of $I$. If in addition
$b''\leq_{\str}b'$, then $\Phi_{\mathbf i}(b')=\Phi_{\mathbf i}(b'')$.
The antisymmetry of $\leq_{\str}$ thus follows from the fact that
the map $\Phi_{\mathbf i}$ separates the points of $B(-\infty)$,
provided $\mathbf i$ has been chosen long enough (see the proof
of Theorem~\ref{th:BerKazh}). Assertion~\ref{it:SOa} is proved.

Assertion~\ref{it:SOb} follows from arguments similar to those
used in the proof of Proposition~\ref{pr:LexOrd}, the key
observation being that given two bases of canonical type
$\mathbf B'$ and $\mathbf B''$, for any
$(b',c')\in(\mathbf B')^2$ and $(b'',c'')\in(\mathbf B'')^2$,
$$(b',b'')\approx(c',c'')\ \Longrightarrow\
\langle(b'')^*,b'\rangle=\langle(c'')^*,c'\rangle.$$
\end{proof}

\begin{other}{Examples}
\label{ex:StrOrder}
\begin{enumerate}
\item
\label{it:ESOa}
If the Cartan matrix $A$ is of type $A_3$, then the order
$\leq_{\str}$ is trivial: for any $b'$ and $b''$ in $B(-\infty)$,
the condition $b'\leq_{\str}b''$ implies $b'=b''$. This readily
follows from Lemma~10.2 in~\cite{BerensteinZelevinsky93} by
induction on the weight.

Combining this result with Proposition~\ref{pr:StrOrder}~\ref{it:SOb},
we see that in type $A_3$, $\mathbf f$ has only one basis
of canonical type. (A similar uniqueness assertion had been
obtained by Berenstein and Zelevinsky for string bases, see
\cite{BerensteinZelevinsky93}, Theorem~9.1.)

Similar results hold in type $A_1$ or $A_2$.
\item
\label{it:ESOb}
Consider now the type $A_4$ with the usual numbering of the vertices
of the Dynkin diagram. One can check that for the elements
$$b'=\tilde e_3^4\tilde e_2^4\tilde e_1^4\tilde
e_4^4\tilde e_3^4\tilde e_2^41\quad\text{and}\quad
b''=\tilde e_2\tilde e_1^3\tilde e_4\tilde e_3^7\tilde
e_2^7\tilde e_1\tilde e_4^3\tilde e_31,$$
one has $\varphi_i(b')<\varphi_i(b'')$ and
$\varphi_i(\sigma(b'))<\varphi_i(\sigma(b''))$ for each $i\in I$;
it follows that $b'<_{\str}b''$.
\item
\label{it:ESOc}
As a last example, we consider the type $A_r$. The word
$$\mathbf i=(1,2,1,3,2,1,\dots,r,r-1,\ldots,2,1)$$
is a reduced decomposition of the longest element in the Weyl group
of $\mathfrak g$. The algebra $\mathbb Q[N]$ is a cluster algebra,
and from the datum of $\mathbf i$, one can construct a seed in
$\mathbb Q[N]$ by the process described in~\cite{GeissLeclercSchroer08},
Theorem~13.2. The cluster monomials built from this seed belong
to the dual semicanonical basis (\cite{GeissLeclercSchroer08},
Theorem~16.1), so are naturally indexed by a subset
$C\subseteq B(-\infty)$. One can show that any element $b\in C$ is
minimal w.r.t.\ the order $\leq_{\str}$.
By Proposition~\ref{pr:StrOrder}~\ref{it:SOb}, this minimality implies
that the element indexed by $b$ in the dual of a basis of canonical
type is independent of the choice of this basis. In other words, the
cluster monomials attached to our seed belong to the dual of any basis
of canonical type. (Reineke had shown that they belong to the dual
of the canonical basis, see Theorem~6.1 in~\cite{Reineke99}.)
\end{enumerate}
\end{other}

\section{Inclusion of MV polytopes}
\label{se:IncMVPol}
In this section, the Cartan matrix $A$ is supposed to be of finite type.

\subsection{Lusztig data}
\label{ss:LusztigData}
Let $W$ be the Weyl group of $\mathfrak g$; it acts on $\mathfrak h$
and on $\mathfrak h^*$ and is generated by the simple reflections
$s_i$, for $i\in I$. We denote by $(\omega_i^\vee)_{i\in I}$ the
basis of $\mathfrak h$ dual to the basis $(\alpha_i)_{i\in I}$ of
$\mathfrak h^*$. A coweight $\gamma\in\mathfrak h$ is said to be a
chamber coweight if it is $W$-conjugated to a $\omega_i^\vee$; we
denote by $\Gamma$ the set of all chamber coweights.

Let $N$ be the number of positive roots; this is also the length of
the longest element $w_0$ in $W$. We denote by $\mathscr X$ the set
of all sequences $\mathbf i=(i_1,\ldots,i_N)$ such that
$w_0=s_{i_1}\cdots s_{i_N}$. An element $\mathbf i\in\mathscr X$
defines a sequence $(\beta_k)$ of positive roots and a sequence
$(\gamma_k)$ of chamber coweights, for $1\leq k\leq N$, as follows:
$$\beta_k=s_{i_1}\cdots s_{i_{k-1}}\alpha_{i_k},\qquad
\gamma_k=-s_{i_1}\cdots s_{i_k}\omega_{i_k}^\vee.$$
It is well known that $(\beta_k)$ is an enumeration
of the positive roots. One easily checks that
$\langle\gamma_k,\beta_\ell\rangle$ is nonnegative if
$k\geq\ell$, nonpositive if $k<\ell$, and is $1$ if $k=\ell$.

Let $v$ be an indeterminate. Let $(d_i)$ be a family of positive
integers such that the matrix $(d_ia_{i,j})$ is symmetric.
For $n\in\mathbb N$ and $i\in I$, we set
$[n]_i=(v^{d_in}-v^{-d_in})/(v^{d_i}-v^{-d_i})$ and
$[n]_i!=[1]_i\cdots[n]_i$. Let $U_q(\mathfrak g)$ be the quantized
enveloping algebra of $\mathfrak g$; it is a $\mathbb Q(v)$-algebra
generated by elements $E_i$, $F_i$ and $K_i$, for $i\in I$, see for
instance~\cite{BerensteinZelevinsky01}, Section~3.1. We define the
divided powers of $E_i$ by $E_i^{(n)}=E_i^n/[n]_i!$. Let
$U_q(\mathfrak n_+)$ be the subalgebra of $U_q(\mathfrak g)$
generated by the elements $E_i$, and let $x\mapsto\overline x$ be
the automorphism of $\mathbb Q$-algebra of $U_q(\mathfrak n_+)$
such that $\overline v=v^{-1}$ and $\overline{E_i}=E_i$.

Let $T_i$ be the automorphism of $U_q(\mathfrak g)$ constructed
by Lusztig and denoted by $T'_{i,-1}$ in \cite{Lusztig91}. For a
fixed $\mathbf i\in\mathscr X$, it is known that when
$\mathbf n=(n_1,\ldots,n_N)$ runs over $\mathbb N^N$, the monomials
$$E_{\mathbf i}^{(\mathbf n)}=E_{i_1}^{(n_1)}\;
T_{i_1}\bigl(E_{i_2}^{(n_2)}\bigr)\;\cdots\;
(T_{i_1}\cdots T_{i_{N-1}})\bigl(E_{i_N}^{(n_N)}\bigr)$$
form a PBW basis of $U_q(\mathfrak n_+)$. In addition, for each
$\mathbf n\in\mathbb N^N$, there is a unique bar-invariant element
\begin{equation}
\label{eq:TransCanPBW}
b_{\mathbf i}(\mathbf n)=\sum_{\mathbf m\in\mathbb N^N}
\zeta_{\mathbf m}^{\mathbf n}E_{\mathbf i}^{(\mathbf m)},
\end{equation}
with $\zeta_{\mathbf n}^{\mathbf n}=1$ and
$\zeta_{\mathbf m}^{\mathbf n}\in v^{-1}\mathbb Z[v^{-1}]$ for
$\mathbf m\neq\mathbf n$. These elements $b_{\mathbf i}(\mathbf n)$
form the canonical basis of $U_q(\mathfrak n_+)$, which does not
depend on the choice of $\mathbf i$~\cite{Lusztig90}.

After specialization at $v=1$ and under the isomorphism
$\mathbf f\cong U(\mathfrak n_+)$, this construction gives a basis
of canonical type of $\mathbf f$. The map $\mathbf n\mapsto
b_{\mathbf i}(\mathbf n)$ can thus be regarded as a parameterization
of $B(-\infty)$ by $\mathbb N^N$. The inverse bijection
$B(-\infty)\to\mathbb N^N$ is called Lusztig datum in
direction~$\mathbf i$; we denote it by
$b\mapsto\mathbf n_{\mathbf i}(b)$.

\subsection{MV polytopes}
\label{ss:MVPol}
To each $b\in B(-\infty)$, one associates its MV polytope $\Pol(b)$;
this is a convex polytope in $\mathfrak h^*$, whose vertices belong
to $Q_+\cap(\wt(b)-Q_+)$. Moreover, the map $b\mapsto\Pol(b)$ is
injective.

The polytope $\Pol(b)$ can be constructed in several ways: as the
image by a moment map of a certain projective variety, called a
Mirkovi\'c-Vilonen cycle (whence the name `MV polytope')
\cite{Anderson03,Kamnitzer07,Kamnitzer10}, or as the
Harder-Narasimhan polytope of a general representation of the
preprojective algebra built from the Dynkin diagram of
$\mathfrak g$~\cite{BaumannKamnitzerTingley11}.

For our purpose however, the most relevant definition of $\Pol(b)$
uses the notion of Lusztig datum. Let $\mathbf i\in\mathscr X$. As
before, we associate to $\mathbf i$ an enumeration $(\beta_k)$ of
the positive roots of $\mathfrak g$. Let $(n_1,\ldots,n_N)=\mathbf
n_{\mathbf i}(b)$ be the Lusztig datum of $b$ in direction $\mathbf i$.
At first sight, the weight
$$\wt(b)-\sum_{t=1}^kn_t\beta_t$$
depends on $\mathbf i$ and $k$. One can however show that it depends
only on $w=s_{i_1}\cdots s_{i_k}$, so it is legitimate to denote it
by $\mu_w(b)$. Then $\Pol(b)$ can be defined as the convex hull of
the weights $\mu_w(b)$, for all $w\in W$.

By \cite{Kamnitzer10}, the normal fan of $\Pol(b)$ is a coarsening
of the Weyl fan in $\mathfrak h$. The weight $\mu_w(b)$ is a vertex
of $\Pol(b)$ and the normal cone to $\Pol(b)$ at $\mu_w(b)$ is
$wC_0$, where
$$C_0=\{\theta\in\mathfrak h\mid\forall i\in I,\,
\langle\theta,\alpha_i\rangle>0\}$$
is the dominant chamber. In particular, the datum of $\Pol(b)$
determines the weights $\mu_w(b)$.

The polytope $\Pol(b)$ can also be described by its facets.
Specifically, one defines $M_\gamma(b)$ for each chamber coweight
$\gamma\in\Gamma$ by the formula
$$M_{w\omega_i^\vee}(b)=\langle w\omega_i^\vee,\mu_w(b)\rangle$$
(one can check that the right-hand side depends only $w\omega_i^\vee$),
and one has
$$M_{w\omega_i^\vee}(b)=\sup(w\omega_i^\vee)(\Pol(b))\quad\text{and}
\quad\Pol(b)=\{x\in\mathfrak h^*\mid\forall\gamma\in\Gamma,\
\langle\gamma,x\rangle\leq M_\gamma(b)\}.$$
The collection $(M_\gamma(b))$ is called the BZ datum of $b$.

\subsection{The polytope order on $B(-\infty)$}
\label{ss:PolOrder}
Given $(b',b'')\in B(-\infty)^2$, we write $b'\leq_{\pol}b''$ if $b'$
and $b''$ have the same weight and if $\Pol(b')\subseteq\Pol(b'')$.
Obviously, the inclusion between the polytopes is equivalent to
the set of equations $M_\gamma(b')\leq M_\gamma(b'')$, for all
$\gamma\in\Gamma$. The relation $\leq_{\pol}$ is an order on
$B(-\infty)$, because the map $b\mapsto\Pol(b)$ is injective.

This order can also be directly expressed in terms of Lusztig data.
Specifically, let $\mathbf i\in\mathscr X$ and define the sequences
of positive roots $(\beta_k)$ and chamber coweights $(\gamma_k)$ as
in Section~\ref{ss:LusztigData}. For $\mathbf n=(n_1,\ldots,n_N)$ in
$\mathbb N^N$, we set $|\mathbf n|=n_1\beta_1+\cdots+n_N\beta_N$;
this is the weight of the PBW monomial $E_{\mathbf i}^{(\mathbf n)}$.
Given another element $\mathbf m=(m_1,\ldots,m_N)$ in $\mathbb N^N$,
we write $\mathbf n\leq_{\mathbf i}\mathbf m$ if
$|\mathbf m|=|\mathbf n|$ and if
$$\sum_{t=1}^k\langle\gamma_k,\beta_t\rangle n_t\leq
\sum_{t=1}^k\langle\gamma_k,\beta_t\rangle m_t$$
for each $k\in\{1,\ldots,N\}$. The relation $\leq_{\mathbf i}$ is
an order on $\mathbb N^N$, less fine than the lexicographic order.

\begin{proposition}
\label{pr:PolOrderPBWOrder}
Let $(b',b'')\in B(-\infty)^2$. Then $b'\leq_{\pol}b''$ if and only if
$\mathbf n_{\mathbf i}(b')\leq_{\mathbf i}\mathbf n_{\mathbf i}(b'')$
for all $\mathbf i\in\mathscr X$.
\end{proposition}
\begin{proof}
Let $b\in B(-\infty)$ and let $\mathbf i\in\mathscr X$. As in
Section~\ref{ss:LusztigData}, the word $\mathbf i$ defines a sequence
$(\beta_k)$ of positive roots and a sequence $(\gamma_k)$ of
chamber coweights. Write $(n_1,\ldots,n_N)=\mathbf n_{\mathbf i}(b)$.
Then for each $k\in\{1,\ldots,N\}$, we have
$$\sum_{t=1}^k\langle\gamma_k,\beta_t\rangle
n_t-\langle\gamma_k,\wt(b)\rangle=
\langle w\omega_{i_k}^\vee,\mu_w(b)\rangle=
M_{w\omega_{i_k}^\vee}(b),$$
where $w=s_{i_1}\cdots s_{i_k}$.

Given $(b',b'')\in B(-\infty)^2$, the inequality
$\mathbf n_{\mathbf i}(b')\leq_{\mathbf i}\mathbf n_{\mathbf i}(b'')$
is thus equivalent to $\wt(b')=\wt(b'')$ and
$M_\gamma(b')\leq M_\gamma(b'')$ for all $\gamma\in\{s_{i_1}\cdots
s_{i_k}\omega_{i_k}^\vee\mid 1\leq k\leq N\}$. The lemma now follows
from the fact that every chamber coweight can be written
as $s_{i_1}\cdots s_{i_k}\omega_{i_k}^\vee$ for suitable $\mathbf i$
and~$k$.
\end{proof}

\begin{other}{Example}
\label{ex:PolOrder}
In type $A_3$, consider the elements $b'=(\tilde e_1\tilde e_3)
\tilde e_2^2(\tilde e_1\tilde e_3)1$ and $b''=\tilde e_2
(\tilde e_1\tilde e_3)^2\tilde e_21$. Both elements are fixed by
the involution $\sigma$. The polytope $\Pol(b')$ is the convex hull of
\begin{multline*}
\{0,\alpha_1,\alpha_3,\alpha_1+\alpha_2,\alpha_1+\alpha_3,
\alpha_2+\alpha_3,2\alpha_1+\alpha_2+\alpha_3,\alpha_1+2\alpha_2+\alpha_3,\\
\alpha_1+\alpha_2+2\alpha_3,2\alpha_1+2\alpha_2+\alpha_3,
\alpha_1+2\alpha_2+2\alpha_3,2\alpha_1+2\alpha_2+2\alpha_3\}.
\end{multline*}
The polytope $\Pol(b'')$ is the convex hull of
\begin{multline*}
\{0,\alpha_1,\alpha_2,\alpha_3,\alpha_1+\alpha_3,2\alpha_1+\alpha_2,
\alpha_2+2\alpha_3,\alpha_1+2\alpha_2+\alpha_3,\\
2\alpha_1+2\alpha_2+\alpha_3,2\alpha_1+\alpha_2+2\alpha_3,
\alpha_1+2\alpha_2+2\alpha_3,2\alpha_1+2\alpha_2+2\alpha_3\}.
\end{multline*}
From there, one easily checks that $\Pol(b')\subset\Pol(b'')$.
Since $b'$ and $b''$ have the same weight, we conclude that
$b'\leq_{\pol}b''$. This example seems to have been first observed
by Kamnitzer.
\end{other}

\begin{other}{Remark}
\label{rk:StabPolOrder}
For any $b\in B(-\infty)$, the MV polytope $\Pol(\sigma(b))$ is
the image of $\Pol(b)$ by the involution $x\mapsto\wt(b)-x$
(\cite{Kamnitzer07}, Theorem~6.2). In addition, $\varphi_i(b)$ is
the first component of $\mathbf n_{\mathbf i}(b)$ whenever
$\mathbf i$ begins by $i$.

It follows that $b'\leq_{\pol}b''$ if and only if
$\sigma(b')\leq_{\pol}\sigma(b'')$, and that if $b'\leq_{\pol}b''$,
then $\varphi_i(b')\leq\varphi_i(b'')$ for all $i\in I$. The order
$\leq_{\pol}$ has thus some remote parentage with $\leq_{\str}$.

One can in fact define another order, weaker than $\leq_{\str}$
and $\leq_{\pol}$, in the following way: $b'\leq b''$ if $b'$ and
$b''$ have the same weight and if for any finite sequence of
elementary moves
$$(b',b'')=(b'_0,b''_0)\approx(b'_1,b''_1)\approx\cdots
\approx(b'_\ell,b''_\ell),$$
one has $b'_\ell\leq_{\pol}b''_\ell$. We believe that the order
$\leq$ is trivial in type $A_4$; indeed, with the help of a computer
running GAP~\cite{GAP4,DeGraaf07}, we checked that in type $A_4$,
the order $\leq$ is trivial in weights up to
$10\alpha_1+16\alpha_2+16\alpha_3+10\alpha_4$.
\end{other}

\subsection{Comparison between the canonical basis and PBW bases}
\label{ss:CompCanPBW}
We come back to the equation~\eqref{eq:TransCanPBW} that defines
the coefficients of the transition matrix between the canonical
basis and a PBW basis. Our aim in this section is to obtain a
necessary condition so that $\zeta_{\mathbf m}^{\mathbf n}\neq0$.
Our result (Corollary~\ref{co:TriangPBWCan}) generalizes Theorem~9.13~(a)
in~\cite{Lusztig90} to arbitrary words $\mathbf i\in\mathscr X$,
not necessarily compatible with a quiver orientation; it complements
Proposition~5.1 in~\cite{DeGraaf02} and Theorem~3.13~(ii)
in~\cite{BeckNakajima04}.

We choose $\mathbf i\in\mathscr X$. Then the PBW monomials
$E_{\mathbf i}^{(\mathbf n)}$ form a basis of $\mathbf f$.
We can also consider the partial order $\leq_{\mathbf i}$
on $\mathbb N^N$.
\begin{lemma}
\label{le:ConvPBW}
Let $\mathbf m_1$, $\mathbf m_2$ and $\mathbf n$ in $\mathbb N^N$
be such that $|\mathbf m_1|+|\mathbf m_2|=|\mathbf n|$. Let
$\mathbf p\in\mathbb N^N$ be such that
$E_{\mathbf i}^{(\mathbf p)}$ appears with a nonzero
coefficient in the expansion of
$E_{\mathbf i}^{(\mathbf m_1)}E_{\mathbf i}^{(\mathbf m_2)}$
on the PBW basis. Let $0\leq k\leq N$ and define
$\mathbf n_1\in\mathbb N^k\times\{0\}^{N-k}$
and $\mathbf n_2\in\{0\}^k\times\mathbb N^{N-k}$ so that
$\mathbf n=\mathbf n_1+\mathbf n_2$. Then
$$\bigl(\mathbf m_1\geq_{\mathbf i}\mathbf n_1\quad\text{and}\quad
\mathbf m_2\geq_{\mathbf i}\mathbf n_2\bigr)\ \Longrightarrow\
\mathbf p\geq_{\mathbf i}\mathbf n.$$
In addition, if one of the inequalities
$\mathbf m_1\geq_{\mathbf i}\mathbf n_1$ and
$\mathbf m_2\geq_{\mathbf i}\mathbf n_2$ is strict,
then $\mathbf p>_{\mathbf i}\mathbf n$.
\end{lemma}
\begin{proof}
We have to check the inequality
\begin{equation}
\label{eq:IOrder}
\sum_{t=1}^\ell\langle\gamma_\ell,\beta_t\rangle n_t\leq
\sum_{t=1}^\ell\langle\gamma_\ell,\beta_t\rangle p_t
\end{equation}
for all $\ell\in\{1,\ldots,N\}$. We decompose each element
$\mathbf q\in\mathbb N^N$ as a sum $\mathbf q'+\mathbf q''$,
with $\mathbf q'\in\mathbb N^\ell\times\{0\}^{N-\ell}$ and
$\mathbf q''\in\{0\}^\ell\times\mathbb N^{N-\ell}$; we then
have $E_{\mathbf i}^{(\mathbf q)}=E_{\mathbf i}^{(\mathbf q')}
E_{\mathbf i}^{(\mathbf q'')}$. With this convention,
$E_{\mathbf i}^{(\mathbf p)}$ appears in the expansion of
$E_{\mathbf i}^{(\mathbf m'_1)}E_{\mathbf i}^{(\mathbf m''_1)}
E_{\mathbf i}^{(\mathbf m'_2)}E_{\mathbf i}^{(\mathbf m''_2)}$.
There is thus $\mathbf q\in\mathbb N^N$ such that
$E_{\mathbf i}^{(\mathbf q)}$ appears in the expansion of
$E_{\mathbf i}^{(\mathbf m''_1)}E_{\mathbf i}^{(\mathbf m'_2)}$ and
$E_{\mathbf i}^{(\mathbf p)}$ appears in the expansion of
$E_{\mathbf i}^{(\mathbf m'_1)}E_{\mathbf i}^{(\mathbf q')}
E_{\mathbf i}^{(\mathbf q'')}E_{\mathbf i}^{(\mathbf m''_2)}$.
The convexity property of PBW bases (Lemma~1
in~\cite{LevendorskiiSoibelman91}, Proposition~7
in~\cite{Beck94}, or Theorem~2.3 in~\cite{Xi99c}) imply then
that $|\mathbf p'|=|\mathbf m'_1|+|\mathbf q'|$ and
$|\mathbf p''|=|\mathbf q''|+|\mathbf m''_2|$.

Assume first that $\ell\leq k$.
Since $\mathbf n_1\leq_{\mathbf i}\mathbf m_1$, we have
$$\bigl\langle\gamma_\ell,|\mathbf n'_1|\bigr\rangle\leq
\bigl\langle\gamma_\ell,|\mathbf m'_1|\bigr\rangle.$$
A fortiori,
$$\bigl\langle\gamma_\ell,|\mathbf n'|\bigr\rangle=
\bigl\langle\gamma_\ell,|\mathbf n'_1|\bigr\rangle\leq
\bigl\langle\gamma_\ell,|\mathbf m'_1|+|\mathbf q'|\bigr\rangle
=\bigl\langle\gamma_\ell,|\mathbf p'|\bigr\rangle,$$
which is exactly \eqref{eq:IOrder}.

Now assume that $\ell>k$.
Since $\mathbf n_2\leq_{\mathbf i}\mathbf m_2$, we have
$$\bigl\langle\gamma_\ell,|\mathbf n'_2|\bigr\rangle\leq
\bigl\langle\gamma_\ell,|\mathbf m'_2|\bigr\rangle.$$
Using $|\mathbf m_2|=|\mathbf n_2|$ and
$|\mathbf p''|=|\mathbf q''|+|\mathbf m''_2|$, we get
$$\bigl\langle\gamma_\ell,|\mathbf n''|\bigr\rangle=
\bigl\langle\gamma_\ell,|\mathbf n''_2|\bigr\rangle\geq
\bigl\langle\gamma_\ell,|\mathbf m''_2|\bigr\rangle\geq
\bigl\langle\gamma_\ell,|\mathbf p''|\bigr\rangle.$$
This last relation is equivalent to \eqref{eq:IOrder},
because $|\mathbf n|=|\mathbf p|$.
\end{proof}

We define coefficients $\omega_{\mathbf m}^{\mathbf n}\in\mathbb Q(q)$
by the expansion
$$\overline{E_{\mathbf i}^{(\mathbf n)}}=\sum_{\mathbf m\in\mathbb N^N}
\omega_{\mathbf m}^{\mathbf n}E_{\mathbf i}^{(\mathbf m)}$$
on the PBW basis. It is known that these coefficients actually belong
to $\mathbb Z[v,v^{-1}]$. The following proposition generalizes
equations~9.12~(a) and~(b) in~\cite{Lusztig90}.
\begin{proposition}
\label{pr:TriangBar}
We have $\omega_{\mathbf n}^{\mathbf n}=1$; moreover,
$\omega_{\mathbf m}^{\mathbf n}\neq0$ implies
$\mathbf m\geq_{\mathbf i}\mathbf n$.
\end{proposition}
\begin{proof}
Let $\mathbf n\in\mathbb N^N$. Suppose first that $\mathbf n$ has
several nonzero entries. We can then find $k\in\{1,\ldots,N-1\}$ such
that one of the first $k$ entries and one of the last $N-k$ entries
of $\mathbf n$ is nonzero. Writing $\mathbf n=\mathbf n_1+\mathbf n_2$
with $\mathbf n_1\in\mathbb N^k\times\{0\}^{N-k}$ and
$\mathbf n_2\in\{0\}^k\times\mathbb N^{N-k}$, we obviously have
$E_{\mathbf i}^{(\mathbf n)}=E_{\mathbf i}^{(\mathbf n_1)}
E_{\mathbf i}^{(\mathbf n_2)}$ and $\overline{E_{\mathbf i}^{(\mathbf
n)}}=\overline{E_{\mathbf i}^{(\mathbf n_1)}}\,\overline{E_{\mathbf
i}^{(\mathbf n_2)}}$. Using Lemma~\ref{le:ConvPBW}, we can
then easily conclude by induction on $|\mathbf n|$.

Now suppose that $\mathbf n$ has a single nonzero entry. Then
$\mathbf n$ is the smallest element of its weight in $\mathbb N^N$.
We can thus write
$$b_{\mathbf i}(\mathbf n)=E_{\mathbf i}^{(\mathbf n)}
+\sum_{\mathbf m>_{\mathbf i}\mathbf n}
\zeta_{\mathbf m}^{\mathbf n}E_{\mathbf i}^{(\mathbf m)},$$
and therefore
$$\overline{E_{\mathbf i}^{(\mathbf n)}}=
E_{\mathbf i}^{(\mathbf n)}+\sum_{\mathbf m>_{\mathbf i}\mathbf n}
\Bigl(\zeta_{\mathbf m}^{\mathbf n}E_{\mathbf i}^{(\mathbf m)}
-\overline{\zeta_{\mathbf m}^{\mathbf n}
E_{\mathbf i}^{(\mathbf m)}}\Bigr).$$
The first case of our reasoning shows that only monomials
$E_{\mathbf i}^{(\mathbf p)}$ with
$\mathbf p\geq_{\mathbf i}\mathbf m>_{\mathbf i}\mathbf n$ appear
in the expansion of $\overline{E_{\mathbf i}^{(\mathbf m)}}$,
which concludes the proof.
\end{proof}

The transition matrix $(\zeta_{\mathbf m}^{\mathbf n})$ between the
PBW basis and the canonical basis can be computed from the matrix
$(\omega_{\mathbf m}^{\mathbf n})$ by the Kazhdan-Lusztig algorithm
(\cite{Lusztig90}, Section~7.11). Proposition~\ref{pr:TriangBar}
then implies:
\begin{corollary}
\label{co:TriangPBWCan}
The coefficients of the transition matrix between the PBW basis and
the canonical basis satisfy
$$\zeta_{\mathbf m}^{\mathbf n}\neq0\ \Longrightarrow\
\mathbf m\geq_{\mathbf i}\mathbf n.$$
\end{corollary}

\section{Further examples}
\label{se:FurEx}
We now consider the case where the generalized Cartan matrix $A$ is
symmetric. Using representations of quivers, Lusztig constructed
two bases of canonical type of $\mathbf f$, which he called the
canonical and the semicanonical bases. In this section, we study
what information the orders $\leq_{\str}$ and $\leq_{\pol}$
convey about these constructions. Our main results are
Proposition~\ref{pr:PolOrderDeg} and Theorem~\ref{th:CompCanSemican}.

\subsection{Background on Lusztig's constructions}
\label{ss:BackLuszCons}
We adopt the notation of \cite{Lusztig91}. The Dynkin diagram is a
graph with vertex set $I$; between two vertices $i$ and $j$, there
are $-a_{i,j}$ edges. Since the graph is without loops, each edge
has two endpoints. Orienting an edge is recognizing one of these
endpoints as the tail and the other one as the head. We denote by
$H$ the set of oriented edges; the tail of $h\in H$ is denoted by
$h'$ and its head by $h''$. In addition, there is a fixed point
free involution $h\mapsto\overline h$ that exchanges tails and
heads. An orientation of our graph is a subset $\Omega\subseteq H$
such that $(\Omega,\overline\Omega)$ is a partition of $H$.

A dimension-vector is a function $\nu\in\mathbb N^I$; we identify
such a $\nu$ with the weight $\sum_{i\in I}\nu(i)\alpha_i$ in $Q_+$.
Given a dimension-vector $\nu$, we denote by $S_\nu$ the set of all
pairs consisting of a sequence $\mathbf i=(i_1,\ldots,i_m)$ of
elements of $I$ and a sequence $\mathbf a=(a_1,\ldots a_m)$ of
natural numbers such that $\nu=\sum_{k=1}^ma_k\alpha_{i_k}$.
Each $(\mathbf i,\mathbf a)\in S_\nu$ defines an element
$\Theta_{\mathbf i}^{(\mathbf a)}=\theta_{i_1}^{(a_1)}\cdots
\theta_{i_m}^{(a_m)}$ in $\mathbf f_\nu$.

Let $\mathbf k$ be an algebraically closed field. To an $I$-graded
$\mathbf k$-vector space $\mathbf V=\bigoplus_{i\in I}\mathbf V_i$,
we associate its dimension-vector $\nu:i\mapsto\dim\mathbf V_i$.
Given $\mathbf V$, let
$G_{\mathbf V}=\prod_{i\in I}\GL(\mathbf V_i)$ and let
$$\mathbf E_{\mathbf V}=\bigoplus_{h\in H}\Hom(\mathbf
V_{h'},\mathbf V_{h''}).$$
The group $G_{\mathbf V}$ acts on the vector space $\mathbf E_{\mathbf V}$
in a natural fashion. An element $x=(x_h)$ in $\mathbf E_{\mathbf V}$
is said to be nilpotent if there exists an $N\geq1$ such that
the composition $x_{h_N}\cdots x_{h_1}:\mathbf V_{h'_1}\to\mathbf
V_{h''_N}$ is zero for each path $(h_1,\ldots,h_N)$ in the oriented
graph. Lastly, given an orientation $\Omega$, let
$\mathbf E_{\mathbf V,\Omega}$ be the subspace of $\mathbf
E_{\mathbf V}$ consisting of all vectors $(x_h)$ such that $x_h=0$
whenever $h\in H\setminus\Omega$.

Fix an $I$-graded $\mathbf k$-vector space $\mathbf V$ of
dimension-vector $\nu$ and an orientation $\Omega$. Let
$(\mathbf i,\mathbf a)\in S_\nu$. By definition, a flag of type
$(\mathbf i,\mathbf a)$ is a decreasing filtration $\mathbf V=
\mathbf V^0\supseteq\mathbf V^1\supseteq\cdots\supseteq\mathbf V^m=0$
of $I$-graded vector spaces such that $\mathbf V^{k-1}/\mathbf V^k$
has dimension vector $a_k\alpha_{i_k}$. Let $\mathcal F_{\mathbf
i,\mathbf a}$ be the set of all flags of type $(\mathbf i,\mathbf a)$
and let $\widetilde{\mathcal F}_{\mathbf i,\mathbf a}$ be the
set of all pairs $(x,\mathbf V^\bullet)\in\mathbf E_{\mathbf V,\Omega}
\times\mathcal F_{\mathbf i,\mathbf a}$ such that each $\mathbf V^k$
is stable by the action of $x$. Let $\pi_{\mathbf i,\mathbf a}:
\widetilde{\mathcal F}_{\mathbf i,\mathbf a}\to\mathbf
E_{\mathbf V,\Omega}$ be the first projection and set
$L_{\mathbf i,\mathbf a;\Omega}=(\pi_{\mathbf i,\mathbf a})_!1$,
where $1$ is the trivial local system on
$\widetilde{\mathcal F}_{\mathbf i,\mathbf a}$. By the Decomposition
Theorem, $L_{\mathbf i,\mathbf a;\Omega}$ is a semisimple complex.

Let $\mathcal P_{\mathbf V,\Omega}$ be the set of isomorphism classes
of simple perverse sheaves $L$ such that $L[d]$ appears as a direct
summand of the sheaf $L_{\mathbf i,\mathbf a;\Omega}$, for some
$(\mathbf i,\mathbf a)\in S_\nu$ and some $d\in\mathbb Z$.
Let $\mathcal Q_{\mathbf V,\Omega}$ be the smallest full subcategory
of the category of bounded complexes of constructible sheaves on
$\mathbf E_{\mathbf V,\Omega}$ that contains the sheaves
$L_{\mathbf i,\mathbf a;\Omega}$ and that is stable by direct sums,
direct summands, and shifts. Lastly, let
$\mathcal K_{\mathbf V,\Omega}$ be the abelian group with one
generator $(L)$ for each isomorphism class of object in
$\mathcal Q_{\mathbf V,\Omega}$, and with relations $(L[1])=(L)$ and
$(L)=(L')+(L'')$ whenever $L$ is isomorphic to $L'\oplus L''$. (Thus
our $\mathcal K_{\mathbf V,\Omega}$ is the specialization at $v=1$ of
the $\mathcal K_{\mathbf V,\Omega}$ defined in Section~10.1
of~\cite{Lusztig91}.)

Given $\nu\in Q_+$, the groups $\mathcal K_{\mathbf V,\Omega}$, for
$\mathbf V$ of dimension-vector $\nu$, can be canonically identified.
We thus obtain an abelian group $\mathcal K_{\nu,\Omega}$ equipped
with isomorphisms $\mathcal K_{\nu,\Omega}\cong\mathcal K_{\mathbf
V,\Omega}$ for any $\mathbf V$ of dimension-vector $\nu$. With
this notation, $\mathcal P_{\mathbf V,\Omega}$ gives rise to a
$\mathbb Z$-basis of $\mathcal K_{\nu,\Omega}$ which does not depend
on $\mathbf V$. We set $\mathcal K_\Omega=\bigoplus_{\nu\in Q_+}
\mathcal K_{\nu,\Omega}$; this is a free $\mathbb Z$-module endowed
with a canonical basis.

We endow $\mathcal K_\Omega$ with the structure of an associative
$Q_+$-graded algebra and we define an isomorphism of algebras
$\lambda_\Omega:\mathbf f\to\mathcal K_\Omega$ as in Sections~10.2 and
10.16 of~\cite{Lusztig91}. Then $\lambda_\Omega\bigl(\Theta_{\mathbf
i}^{(\mathbf a)}\bigr)= (L_{\mathbf i,\mathbf a;\Omega})$ for each
$(\mathbf i,\mathbf a)\in S_\nu$.

By Theorem~10.17 in~\cite{Lusztig91}, the inverse image of the
canonical basis of $\mathcal K_\Omega$ by $\lambda_\Omega$ does
not depend on $\Omega$. This inverse image is called the canonical
basis of $\mathbf f$; it is a basis of canonical type. By
Section~\ref{ss:CrysBInfinity}, there is thus a unique isomorphism
of crystals from $B(-\infty)$ onto the canonical basis; following
Kashiwara, we denote this isomorphism by $G$. In the sequel, given
$b\in B(-\infty)$ and $\mathbf V$ of dimension-vector $\wt(b)$, we
denote by $L_{b,\Omega}$ the element in $\mathcal P_{\mathbf V,\Omega}$
such that $(L_{b,\Omega})=\lambda_\Omega(G(b))$ in
$\mathcal K_{\mathbf V,\Omega}$.

For an orientation $\Omega$ and for $h\in H$,
we set $\varepsilon_\Omega(h)=1$ if $h\in\Omega$ and
$\varepsilon_\Omega(h)=-1$ otherwise.
Given an $I$-graded $\mathbf k$-vector space $\mathbf V$, let
$\Lambda_{\mathbf V,\Omega}$ be the set of all nilpotent elements
$x=(x_h)$ in $\mathbf E_{\mathbf V}$ such that for each $i\in I$,
$$\sum_{\substack{h\in H\\h''=i}}\varepsilon_\Omega(h)
x_hx_{\overline h}=0.$$
This set $\Lambda_{\mathbf V,\Omega}$ is called the nilpotent variety.

Up to a canonical bijection, the set of irreducible components of
$\Lambda_{\mathbf V,\Omega}$ depends only on the dimension-vector
$\nu$ of $\mathbf V$, and not on $\mathbf V$ or $\Omega$
(\cite{Lusztig91}, Sections~12.14 and~12.15). We denote this set
by $\mathbf Z_\nu$ and we set $\mathbf Z=\bigsqcup_{\nu\in Q_+}\mathbf
Z_\nu$. Then $\mathbf Z$ can be endowed with the structure of a crystal
isomorphic to $B(-\infty)$ (\cite{KashiwaraSaito97}, Theorem~5.3.2).
Given $b\in B(-\infty)$ and $\mathbf V$ of dimension-vector $\wt(b)$,
we denote by $\Lambda_{b,\Omega}$ the irreducible component of
$\Lambda_{\mathbf V,\Omega}$ that corresponds to $b$.

Up to a canonical isomorphism, the space of $\mathbb Q$-valued,
$G_{\mathbf V}$-invariant, constructible fonctions on
$\Lambda_{\mathbf V,\Omega}$ depends only on the dimension-vector
$\nu$ of $\mathbf V$; we denote it by $\widetilde M(\nu)$.
In Section~12 of~\cite{Lusztig91}, Lusztig endows
$\widetilde M=\bigoplus_{\nu\in Q_+}\widetilde M(\nu)$ with
the structure of an algebra and constructs an injective
homomorphism $\kappa:\mathbf f\to\widetilde M$ (this morphism
is denoted by $\gamma$ in \cite{Lusztig91} and by $\kappa$
in~\cite{Lusztig00}). The map $\kappa$ is defined so that
for each $(\mathbf i,\mathbf a)\in S_\nu$, the value of
$\kappa\bigl(\Theta_{\mathbf i}^{(\mathbf a)}\bigr)$ at a
point $x\in\Lambda_{\mathbf V,\Omega}$ is the Euler
characteristic of the set of all $x$-stable flags of type
$(\mathbf i,\mathbf a)$ in $\mathbf V$.

Each irreducible component $X$ of $\Lambda_{\mathbf V,\Omega}$
contains a dense open subset $X_0$ such that any function in
$\kappa(\mathbf f_\nu)$ is constant on $X_0$. We denote by
$\delta_X:\mathbf f_\nu\to\mathbb Q$ the linear form obtained
by composing $\kappa$ with the evaluation at a point of $X_0$.
By Section~12.14 in \cite{Lusztig91} or by Theorem~2.7
in~\cite{Lusztig00}, the elements $\delta_X$, for
$X\in\mathbf Z_\nu$, form a basis of the dual of
$\mathbf f_\nu$. Gathering the corresponding dual bases in
$\mathbf f_\nu$ for all $\nu\in Q_+$, we get a basis of $\mathbf f$.
This is the semicanonical basis, and it is of canonical
type. There is thus a unique isomorphism of crystals from
$B(-\infty)$ onto the semicanonical basis; we denote this
isomorphism by $S$. Comparing the constructions
in~\cite{KashiwaraSaito97} and in~\cite{Lusztig00}, one
checks that for any $b\in B(-\infty)$, the dual vector to
$S(b)$ is the $\delta_X$ with $X=\Lambda_{b,\Omega}$. In
other words, the indexations of the semicanonical basis and
of $\mathbf Z$ by $B(-\infty)$ agree.

\subsection{Condition for singular supports}
\label{ss:SingSupp}
In the context of Section~\ref{ss:BackLuszCons}, the trace duality
allows us to regard $\mathbf E_{\mathbf V}$ as the cotangent space
of $\mathbf E_{\mathbf V,\Omega}$. Lusztig proved (\cite{Lusztig91},
Corollary~13.6) that the singular support of a complex in
$\mathcal Q_{\mathbf V,\Omega}$ is the union of irreducible
components of $\Lambda_{\mathbf V,\Omega}$. In this context,
Kashiwara and Saito (\cite{KashiwaraSaito97}, Lemma~8.2.1) proved that
$$\Lambda_{b'',\Omega}\subseteq SS(L_{b',\Omega})\ \Longrightarrow\
b'\leq_{\str}b''.$$

\subsection{Degeneracy order between quiver representations}
\label{ss:DegOrder}
We continue with the case where $A$ is a symmetric Cartan matrix
and assume in addition that $A$ is of finite type. We fix $\nu\in Q_+$
and an $I$-graded $\mathbf k$-vector space $\mathbf V$ of
dimension-vector $\nu$. For each orientation $\Omega$ and each element
$b\in B(-\infty)$ of weight $\nu$, there is a $G_{\mathbf V}$-orbit
$\mathscr O_{b,\Omega}\subseteq\mathbf E_{\mathbf V,\Omega}$
such that the perverse sheaf $L_{b,\Omega}$ is the intersection
cohomology sheaf on $\overline{\mathscr O_{b,\Omega}}$ w.r.t.\ the
trivial local system on $\mathscr O_{b,\Omega}$.

\begin{proposition}
\label{pr:PolOrderDeg}
Let $b'$ and $b''$ be two elements of weight $\nu$ in $B(-\infty)$.
If $b'\leq_{\pol}b''$, then
$\overline{\mathscr O_{b',\Omega}}\supseteq\mathscr O_{b'',\Omega}$
for all orientations $\Omega$.
\end{proposition}
\begin{proof}
Fix an orientation $\Omega$. For $\mu$ and $\nu$ in $Q_+$, define
$$\langle\mu,\nu\rangle_\Omega=\sum_{i\in I}\mu_i\nu_i-\sum_{h\in\Omega}
\mu_{h'}\nu_{h''}.$$
The oriented graph $Q=(I,\Omega)$ is a Dynkin quiver. Given a
positive root $\alpha$, we denote by $M(\alpha)$ the
indecomposable $\mathbf kQ$-module of dimension-vector $\alpha$.
Ringel (\cite{Ringel96}, p.~59) has shown that
\begin{equation}
\label{eq:RingelFormulas}
\begin{aligned}
\dim\Hom_{\mathbf kQ}(M(\alpha),M(\beta))&=\max\bigl(0,
\langle\alpha,\beta\rangle_\Omega\bigr),\\[2pt]
\dim\Ext^1_{\mathbf kQ}(M(\alpha),M(\beta))&=\max\bigl(0,
-\langle\alpha,\beta\rangle_\Omega\bigr).
\end{aligned}
\end{equation}

Choose $\mathbf i\in\mathscr X$ adapted to $\Omega$
(\cite{Lusztig90}, Section~4.7 and Proposition~4.12~(b)). As
in Section~\ref{ss:LusztigData}, the word $\mathbf i$ defines a
sequence $(\beta_k)$ of positive roots and a sequence $(\gamma_k)$
of chamber coweights. By Proposition~7.4 in~\cite{BaumannKamnitzer12},
we have, for any $k\in\{1,\ldots,N\}$,
$$\langle\gamma_k,?\rangle=\langle\beta_k,?\rangle_{\overline\Omega}
=\langle?,\beta_k\rangle_\Omega.$$
It follows that for any $k$ and $\ell$ in $\{1,\ldots,N\}$,
we have
$$\dim\Hom_{\mathbf kQ}(M(\beta_\ell),M(\beta_k))=
\max\bigl(0,\langle\gamma_k,\beta_\ell\rangle\bigr)=
\begin{cases}
\langle\gamma_k,\beta_l\rangle&\text{if }k\geq\ell,\\
0&\text{if }k<\ell.
\end{cases}$$

For $b\in B(-\infty)$ of weight $\nu$, the $G_{\mathbf V}$-orbit
$\mathscr O_{b,\Omega}$ is the set of all
$x\in\mathbf E_{\mathbf V,\Omega}$ such that
$$(\mathbf V,x)\cong M(\beta_1)^{\oplus n_1}\oplus\cdots\oplus
M(\beta_N)^{\oplus n_N}$$
as $\mathbf kQ$-modules, where $(n_1,\ldots,n_N)$ is the Lusztig
datum of $b$ in direction $\mathbf i$ (\cite{Lusztig90},
Sections~4.15--4.16).

Assume that $b'\leq_{\pol}b''$. We then have
$\mathbf n_{\mathbf i}(b')\leq_{\mathbf i}\mathbf n_{\mathbf i}(b'')$.
This inequality is equivalent to the fact that for each
$k\in\{1,\ldots,N\}$,
$$\dim\Hom_{\mathbf kQ}((\mathbf V,x'),M(\beta_k))\leq
\dim\Hom_{\mathbf kQ}((\mathbf V,x''),M(\beta_k)),$$
where $x'\in\mathscr O_{b',\Omega}$ and $x''\in\mathscr O_{b'',\Omega}$.
The inclusion $\overline{\mathscr O_{b',\Omega}}\supseteq
\mathscr O_{b'',\Omega}$ now follows from Riedtmann's
criterion~\cite{Riedtmann86,Bongartz95}.
\end{proof}

\begin{other}{Remark}
\label{rk:DegOrder}
\begin{enumerate}
\item
\label{it:DOa}
The converse statement to Proposition~\ref{pr:PolOrderDeg} is true
in type $A$. This comes from the fact that in this case, any chamber
coweight can be written as $s_{i_1}\cdots s_{i_k}\omega_{i_k}^\vee$,
where $\mathbf i\in\mathscr X$ is compatible with an orientation
(see the proof of Proposition~A.2 in~\cite{BaumannKamnitzerTingley11}).
\item
\label{it:DOb}
Let us go back to the problem studied in Section~\ref{ss:SingSupp}.
Let $b'$ and $b''$ in $B(-\infty)$ have the same weight. Given an
orientation $\Omega$, it is known that $\Lambda_{b'',\Omega}$ is
the closure of the conormal bundle to $\mathcal O_{b'',\Omega}$
(see \cite{Kimura07}, Section~5.3 or \cite{BaumannKamnitzer12},
Section~7.3). Therefore, in order that
$\Lambda_{b'',\Omega}\subseteq SS(L_{b',\Omega})$, it is necessary that
$\mathcal O_{b'',\Omega}\subseteq\overline{\mathcal O_{b',\Omega}}$.
Since the condition $\Lambda_{b'',\Omega}\subseteq SS(L_{b',\Omega})$
does not depend on $\Omega$ (\cite{KashiwaraSaito97}, Theorem~6.2.1),
we must in fact have
$\overline{\mathcal O_{b',\Omega}}\supseteq\mathcal O_{b'',\Omega}$
for all orientations $\Omega$. In type $A$, this means that
$b'\leq_{\pol}b''$ by the previous remark. We do not know if this
result extends to the other types.
\end{enumerate}
\end{other}

\subsection{Comparison between the canonical and the semicanonical bases}
\label{ss:CompCanSemican}
We continue to assume that the Cartan matrix $A$ is symmetric and of
finite type. As both the canonical and the semicanonical bases are of
canonical type, the transition matrix between them is lower unitriangular
w.r.t.\ the order $\leq_{\str}$ (see Proposition~\ref{pr:StrOrder}). Our
aim now is to compare these bases w.r.t.\ the order $\leq_{\pol}$.

Our method is to use a PBW basis as an intermediary. Conditions on
the transition matrix between the canonical basis and a PBW basis
were obtained in Corollary~\ref{co:TriangPBWCan}. We now focus on
the transition matrix between the semicanonical basis and a PBW basis.

\begin{lemma}
\label{le:TriangPBWSemican}
Let $\mathbf i\in\mathscr X$, $\mathbf n\in\mathbb N^N$ and
$b\in B(-\infty)$. If $S(b)$ appears with a nonzero coefficient
in the expansion of $E_{\mathbf i}^{(\mathbf n)}$ on the
semicanonical basis, then
$\mathbf n_{\mathbf i}(b)\geq_{\mathbf i}\mathbf n$.
\end{lemma}
\begin{proof}
We can of course assume that $E_{\mathbf i}^{(\mathbf n)}$ and
$b$ have the same weight; call it $\nu$. As before, $\mathbf i$
defines a sequence $(\beta_k)$ of positive roots and a sequence
$(\gamma_k)$ of chamber coweights.

The coefficient of $S(b)$ in the expansion of an element
$u\in\mathbf f_\nu$ on the semicanonical basis is equal to
$\delta_x(u)=\kappa(u)(x)$, where $x$ is a general point in
$\Lambda_{b,\Omega}$. By definition of the algebra structure on
$\widetilde M$, if this number is nonzero for
$u=E_{\mathbf i}^{(\mathbf n)}$, then there is a filtration
$$\mathbf V=\mathbf V_0\supseteq\mathbf V_1\supseteq\cdots\supseteq
\mathbf V_{N-1}\supseteq\mathbf V_N=0$$
by $x$-stable $I$-graded vector spaces such that the dimension-vector
of $\mathbf V_{k-1}/\mathbf V_k$ is $n_k\beta_k$.

Theorem~6.3 in~\cite{BaumannKamnitzer12} tells us how to extract
the BZ datum of $b$ from the pair $(\mathbf V,x)$, viewed as a
representation of the preprojective algebra $\Pi$. Specifically,
if we introduce the $\Pi$-modules $N(\gamma)$ as in Section~3.4 of
\cite{BaumannKamnitzer12}, then
$M_\gamma(b)=\dim\Hom_\Pi(N(\gamma),(\mathbf V,x))$.

Let $k\in\{1,\ldots,N\}$ and set $w=s_{i_1}\cdots s_{i_k}$.
Observing that $(\mathbf V_k,x)$ is a $\Pi$-submodule of
$(\mathbf V,x)$ and using Proposition~4.3
in~\cite{BaumannKamnitzer12}, we get
$$\langle w\omega_{i_k}^\vee,\mu_w(b)\rangle
=M_{w\omega_{i_k}^\vee}(b)
\geq\dim\Hom_\Pi(N(w\omega_{i_k}^\vee),(\mathbf V_k,x))
\geq\langle w\omega_{i_k}^\vee,\nu_k\rangle,$$
where $\nu_k$ is the dimension-vector of $\mathbf V_k$.
Substituting $\nu_k=\wt(b)-\sum_{t=1}^kn_t\beta_t$ and
$w\omega_{i_k}^\vee=-\gamma_k$, we get exactly the
inequality asked for in the definition of
$\mathbf n_{\mathbf i}(b)\geq_{\mathbf i}\mathbf n$.
\end{proof}

\begin{theorem}
\label{th:CompCanSemican}
The transition matrix between the canonical and the semicanonical
bases is lower unitriangular w.r.t.\ the order $\leq_{\pol}$.
\end{theorem}
\begin{proof}
Proposition~\ref{pr:PolOrderPBWOrder}, Corollary~\ref{co:TriangPBWCan}
and Lemma~\ref{le:TriangPBWSemican} readily imply that the transition
matrix is lower triangular w.r.t.\ the order $\leq_{\pol}$. In
addition, Proposition~\ref{pr:StrOrder}~\ref{it:SOb} guarantees that
the diagonal coefficients are equal to $1$.
\end{proof}

\begin{other}{Remark}
\label{rk:CompCanSemican}
The proof of Proposition~\ref{pr:StrOrder} relies on the observation
that the moves $\approx$ preserve the coefficients of the transition
matrix between two bases of canonical type. Using
Theorem~\ref{th:CompCanSemican}, we then see that the transition
matrix between the canonical and the semicanonical bases is also
lower triangular w.r.t.\ the order $\leq$ defined in
Remark~\ref{rk:StabPolOrder}, and obtained by ``stabilizing''
$\leq_{\pol}$ under the moves $\approx$. In addition, this
transition matrix is also invariant under Saito's crystal reflections
(see~\cite{Baumann11}, equation~(3)); we can thus weaken again our
order by introducing a further move. (The author borrowed this idea
from Kashiwara and Saito, see Lemma~8.2.2 in~\cite{KashiwaraSaito97}.)
\end{other}

\section{A study in types $A_5$ and $D_4$}
\label{se:StudA5D4}
In~\cite{KashiwaraSaito97}, Kashiwara and Saito discovered a situation
where $\Lambda_{b'',\Omega}\subseteq SS(L_{b',\Omega})$ for two
elements $b'\neq b''$ of $B(-\infty)$, in the notation of Section
\ref{ss:SingSupp}. In \cite{GeissLeclercSchroer05}, Geiss, Leclerc
and Schr\"oer made a more detailed investigation of this situation,
and observed that for these elements, the element $S(b'')$ do occur
in the expansion of the canonical basis element $G(b')$ on the
semicanonical basis. By Proposition~\ref{pr:StrOrder}~\ref{it:SOb}
and Theorem~\ref{th:CompCanSemican}, it follows that
$b'\leq_{\str}b''$ and $b'\leq_{\pol}b''$.

Geiss, Leclerc and Schr\"oer explain that these phenomena are related
to the fact that the algebra $\mathbb Q[N]$ has a tame cluster type,
namely $E_8^{(1,1)}$, and that the counterexample is located precisely at
one of the imaginary indecomposable roots of this elliptic root system.
Leclerc explained to the author that the other imaginary root gives
rise to a similar counterexample, and that the story can be repeated
word for word in type $D_4$.

Our aim here is to add a small piece to the almost complete description
of the situation given in~\cite{GeissLeclercSchroer05}: we will
compute explicitly the elements $G(b')$ and $G(b'')$ above, in a way
that will allow us in the next section to compute the image of
$G(b')$ by the quantum Frobenius map and the quantum Frobenius
splitting.

\subsection{Statement of the results}
\label{ss:StatRes}
Our results can be stated in an uniform way for four situations,
numbered (I)--(IV). In each case, and for each $p\geq1$, we define
a product $\widetilde E_p$ of crystal operators and elements $\xi_p$
and $\eta_p$ in $\mathbf f$, as in the following table. The first two
cases are in type $A_5$, with the vertices of the Dynkin diagram
sequentially numbered. The last two cases are in type $D_4$, where
the index of the central node is $2$.

\begin{center}
\begin{tabular}{ll}
\toprule
I&Type $A_5$\\[4pt]
&$\widetilde E_p=(\tilde e_2\tilde e_4)^p\;
(\tilde e_1\tilde e_3^2\tilde e_5)^p\;
(\tilde e_2\tilde e_4)^p$\\[4pt]
&$\xi_p=\bigl(\theta_2^{(p)}\theta_4^{(p)}\bigr)\;
\bigl(\theta_1^{(p)}\theta_3^{(2p)}\theta_5^{(p)}\bigr)\;
\bigl(\theta_2^{(p)}\theta_4^{(p)}\bigr)$\\[4pt]
&$\eta_p=\bigl(\theta_2^{(p)}\theta_4^{(p)}\bigr)\;
\bigl(\theta_1^{(p)}\theta_3^{(2p)}\theta_5^{(p)}\bigr)\;
\bigl(\theta_2^{(2p)}\theta_4^{(2p)}\bigr)\;
\bigl(\theta_1^{(p)}\theta_3^{(2p)}\theta_5^{(p)}\bigr)\;
\bigl(\theta_2^{(p)}\theta_4^{(p)}\bigr)$\\[4pt]
\midrule
II&Type $A_5$\\[4pt]
&$\widetilde E_p=(\tilde e_1\tilde e_3^2\tilde e_5)^p\;
(\tilde e_2\tilde e_4)^{3p}\;
(\tilde e_1\tilde e_3^2\tilde e_5)^p$\\[4pt]
&$\xi_p=\bigl(\theta_1^{(p)}\theta_3^{(2p)}\theta_5^{(p)}\bigr)\;
\bigl(\theta_2^{(3p)}\theta_4^{(3p)}\bigr)\;
\bigl(\theta_1^{(p)}\theta_3^{(2p)}\theta_5^{(p)}\bigr)$\\[4pt]
&$\eta_p=\bigl(\theta_1^{(p)}\theta_3^{(2p)}\theta_5^{(p)}\bigr)\;
\bigl(\theta_2^{(3p)}\theta_4^{(3p)}\bigr)\;
\bigl(\theta_1^{(2p)}\theta_3^{(4p)}\theta_5^{(2p)}\bigr)\;
\bigl(\theta_2^{(3p)}\theta_4^{(3p)}\bigr)\;
\bigl(\theta_1^{(p)}\theta_3^{(2p)}\theta_5^{(p)}\bigr)$\\[4pt]
\midrule
III&Type $D_4$\\[4pt]
&$\widetilde E_p=\tilde e_2^p\;
(\tilde e_1\tilde e_3\tilde e_4)^p\;
\tilde e_2^p$\\[4pt]
&$\xi_p=\theta_2^{(p)}\;
\bigl(\theta_1^{(p)}\theta_3^{(p)}\theta_4^{(p)}\bigr)\;
\theta_2^{(p)}$\\[4pt]
&$\eta_p=\theta_2^{(p)}\;
\bigl(\theta_1^{(p)}\theta_3^{(p)}\theta_4^{(p)}\bigr)\;
\theta_2^{(2p)}\;
\bigl(\theta_1^{(p)}\theta_3^{(p)}\theta_4^{(p)}\bigr)\;
\theta_2^{(p)}$\\[4pt]
\midrule
IV&Type $D_4$\\[4pt]
&$\widetilde E_p=(\tilde e_1\tilde e_3\tilde e_4)^p\;
\tilde e_2^{3p}\;
(\tilde e_1\tilde e_3\tilde e_4)^p$\\[4pt]
&$\xi_p=\bigl(\theta_1^{(p)}\theta_3^{(p)}\theta_4^{(p)}\bigr)\;
\theta_2^{(3p)}\;
\bigl(\theta_1^{(p)}\theta_3^{(p)}\theta_4^{(p)}\bigr)$\\[4pt]
&$\eta_p=\bigl(\theta_1^{(p)}\theta_3^{(p)}\theta_4^{(p)}\bigr)\;
\theta_2^{(3p)}\;
\bigl(\theta_1^{(2p)}\theta_3^{(2p)}\theta_4^{(2p)}\bigr)\;
\theta_2^{(3p)}\;
\bigl(\theta_1^{(p)}\theta_3^{(p)}\theta_4^{(p)}\bigr)$\\[4pt]
\bottomrule
\end{tabular}
\end{center}

For each $(r,s)\in\mathbb N^2$, we set
$b_{r,s}=\widetilde E_{r+s}\widetilde E_r1$; this is an element
in $B(-\infty)$.

\begin{proposition}
\label{pr:SmallA5D4}
Let $r\in\mathbb N$. Then $b_{r,0}$ is maximal in $B(-\infty)$
w.r.t.\ the order $\leq_{\str}$. Further, if $\mathbf B$ is a basis
of canonical type of $\mathbf f$, then $\xi_r$ is the element of
$\mathbf B$ indexed by $b_{r,0}$ in the bijection
$B(-\infty)\cong\mathbf B$.
\end{proposition}
Proposition~\ref{pr:SmallA5D4} is proved in Section
\ref{ss:ProofProp}.

Recall that $G(b_{r,s})$ and $S(b_{r,s})$ denote the elements
indexed by $b_{r,s}$ in the canonical and semicanonical bases
of $\mathbf f$. Proposition~\ref{pr:SmallA5D4} tells us that
$\xi_p=G(b_{p,0})=S(b_{p,0})$. We now look for similar expansions
of $\eta_p$ on the two bases, canonical and semicanonical.

\begin{theorem}
\label{th:MainA5D4}
\begin{enumerate}
\item
\label{it:MADa}
Let $(r',s',r'',s'')\in\mathbb N^4$. Then
$$b_{r',s'}\leq_{\str}b_{r'',s''}\ \Longleftrightarrow\
b_{r',s'}\leq_{\pol}b_{r'',s''}\ \Longleftrightarrow\
\bigl(r'+2s'=r''+2s''\;\text{ and }\;r'\leq r''\bigr).$$
\item
\label{it:MADb}
For each $p\in\mathbb N$,
$$\eta_p=G(b_{0,p})+G(b_{2,p-1})+G(b_{4,p-2})+\cdots+G(b_{2p,0}).$$
\item
\label{it:MADc}
For each $(r,s)\in\mathbb N^2$,
$$\langle S(b_{2r,s})^*,\eta_{r+s}\rangle=\binom{2r}r.$$
\end{enumerate}
\end{theorem}

The proof of Theorem~\ref{th:MainA5D4} occupies
Sections~\ref{ss:ProofTh(i)}--\ref{ss:ProofTh(ii)}.

\begin{other}{Remark}
\label{rk:CexKS}
Recall the notation of Section~\ref{ss:BackLuszCons}.
Let $\nu\in Q_+$, let $\Omega$ be an orientation, let $\mathbf V$
be an $I$-graded $\mathbf k$-vector space of dimension-vector $\nu$,
and let $(\mathbf i,\mathbf a)\in S_\nu$. By Theorem~2.2
in~\cite{Sabbah85}, the characteristic cycle of
$L_{\mathbf i,\mathbf a;\Omega}$ is the projection on
$T^*\mathbf E_{\mathbf V,\Omega}$ of the intersection in the
ambient space $T^*(\mathcal F_{\mathbf i,\mathbf a}\times\mathbf
E_{\mathbf V,\Omega})$ of $\mathcal F_{\mathbf i,\mathbf a}\times
T^*\mathbf E_{\mathbf V,\Omega}$ with the conormal bundle of
$\widetilde{\mathcal F}_{\mathbf i,\mathbf a}$. (In this recipe,
$\mathcal F_{\mathbf i,\mathbf a}$ is identified to the zero
section of $T^*\mathcal F_{\mathbf i,\mathbf a}$ and intersection
means the intersection product in homology or in algebraic geometry.)

Let $p\in\mathbb N$ and take $(\mathbf i,\mathbf a)$ such that
$\Theta_{\mathbf i}^{(\mathbf a)}=\eta_p$. Looking over a general
point of $\Lambda_{b_{2r,s}}$, where $r+s=p$, the intersection
is transversal and consists of $\binom{2r}r$ points.
(These are the same points as those occurring in the
proof of Theorem~\ref{th:MainA5D4}~\ref{it:MADc}.)
The multiplicity of $[\Lambda_{b_{2r,s}}]$ in the characteristic
cycle of $L_{\mathbf i,\mathbf a;\Omega}$ is therefore equal to
this binomial coefficient, up to a sign.

Take $p=1$. The multiplicity of $[\Lambda_{b_{2,0}}]$ in the
characteristic cycle of $L_{\mathbf i,\mathbf a;\Omega}$ is
thus $\pm2$. A similar calculation shows that the multiplicity
of $[\Lambda_{b_{2,0}}]$ in the characteristic cycle of
$L_{b_{2,0},\Omega}$ is equal to $\pm1$. In addition,
the proof of Theorem~\ref{th:MainA5D4}~\ref{it:MADb} gives
$L_{\mathbf i,\mathbf a;\Omega}=L_{b_{0,1},\Omega}\oplus
L_{b_{2,0},\Omega}$. We conclude that $\Lambda_{b_{2,0}}$ is
contained in the singular support of $L_{b_{0,1},\Omega}$.

The elements $b$ and $b'$ of Kashiwara and Saito
(\cite{KashiwaraSaito97}, Section~7.2) are our elements $b_{0,1}$
and $b_{2,0}$ in case~I. The arguments above thus provide a new
proof of Theorem~7.2.1 in~\cite{KashiwaraSaito97}.
\end{other}

\subsection{Proof of Proposition~\ref{pr:SmallA5D4}}
\label{ss:ProofProp}
Though each case requires its own set of calculations, we will
content ourselves with treating the case~I.

We fix $r\in\mathbb N$. Standard tools for crystal combinatorics
allow to compute the Lusztig data of $b_{r,0}$ w.r.t.\ the
reduced words
\begin{align*}
\mathbf i&=(2,4,1,3,5,2,4,1,3,5,2,4,1,3,5),\\
\mathbf j&=(1,3,5,2,4,1,3,5,2,4,1,3,5,2,4).
\end{align*}
One finds
\begin{align*}
\mathbf n_{\mathbf i}(b_{r,0})&=(r,r,0,0,0,0,0,0,0,0,r,r,0,0,0),\\
\mathbf n_{\mathbf j}(b_{r,0})&=(0,0,0,r,r,0,0,0,0,0,0,0,0,r,r).
\end{align*}
One deduces that
$$\varphi_i(b_{r,0})=\varphi_i(\sigma(b_{r,0}))=
\begin{cases}
r&\text{if $i\in\{2,4\}$,}\\
0&\text{if $i\in\{1,3,5\}$.}
\end{cases}$$

We claim that an element $b\in B(-\infty)$ such that
$\wt(b)=\wt(b_{r,0})$ and such that
\begin{equation}
\label{eq:Br0Maximal}
\varphi_i(b)\geq\varphi_i(b_{r,0})\quad\text{and}\quad
\varphi_i(\sigma(b))\geq\varphi_i(\sigma(b_{r,0}))
\end{equation}
for each $i\in\{1,\ldots,5\}$ is necessarily equal to $b_{r,0}$.

To prove this claim, pick such a $b$ and set $t=\varphi_2(b)$,
$u=\varphi_2(\sigma(b))$,
$b'=\sigma\bigl(\tilde f_2^{\max}\sigma(b)\bigr)$.
By Proposition~5.3.1~(1) in~\cite{KashiwaraSaito97}, we have
$$t=\max\bigl(\varphi_2(b'),u+\langle\alpha_2^\vee,\wt(b')\rangle\bigr)
=\max\bigl(\varphi_2(b'),r-u\bigr).$$
Since $t\geq r$ and $u\geq r$, this forces $\varphi_2(b')=t$.
Therefore, in its expansion on the basis of simple roots, the
$\alpha_2$-coordinate of $\wt(b')$ is thus at least $t$, so
$2r-u\geq t$. This forces $t=u=r$. By symmetry, we have
$\varphi_4(b)=\varphi_4(\sigma(b))=r$. Further, the crystal
operations at vertex $2$ commute with the crystal operations at vertex
$4$, so the element $b''=\sigma\bigl(\tilde f_4^r\sigma(b')\bigr)$
satisfies $\varphi_2(b'')=\varphi_4(b'')=r$.

Now $b'''=\tilde f_1^r\tilde f_3^{2r}\tilde f_5^r1$ is the unique
element in $B(-\infty)$ with weight $r(\alpha_1+2\alpha_3+\alpha_5)$.
We thus necessarily have $b'''=\tilde f_2^r\tilde f_4^rb''$,
and we conclude that $b=\sigma\circ\bigl(\tilde f_2^r\tilde f_4^r\bigr)
\circ\sigma\circ\bigl(\tilde f_2^r\tilde f_4^r\bigr)(b''')$.
This reasoning holds in particular for $b_{r,0}$, whence our
claim that $b=b_{r,0}$.

From this, it immediately follows that $b_{r,0}$ is maximal in
$B(-\infty)$ w.r.t.\ the order $\leq_{\str}$. Therefore, all bases
of canonical type share the same element at this spot, by
Proposition~\ref{pr:StrOrder}~\ref{it:SOb}. It remains to prove
that this element is $\xi_r$.

Let us adopt the setup of the definition of the semicanonical
basis. Pick an orientation of the Dynkin diagram and an $I$-graded
$\mathbf k$-vector space $\mathbf V$ of dimension-vector
$\wt(b_{r,0})$. Let $b\in B(-\infty)$. By construction (see
Section~\ref{ss:BackLuszCons}), the value of the constructible
function $\kappa(\xi_r)$ at the general point $x$ of
$\Lambda_{b,\Omega}$ is nonzero only if $\mathbf V$ contains
$x$-stable subspaces $\mathbf W'$ and $\mathbf W''$ such that
$\mathbf W'$ and $\mathbf V/\mathbf W''$ have dimension-vector
$r(\alpha_2+\alpha_4)$. This latter condition is equivalent to
\eqref{eq:Br0Maximal}, hence to $b=b_{r,0}$; when it is fulfilled,
$\mathbf W'$ and $\mathbf W''$ are unique and satisfy
$\mathbf W'\subseteq\mathbf W''$, which leads to
$\kappa(\xi_r)(x)=1$.

To sum up, $\langle S(b)^*,\xi_r\rangle$ is equal to $1$ if
$b=b_{r,0}$ and to $0$ otherwise. We conclude that $\xi_r=S(b_{r,0})$,
as announced.

\subsection{Proof of Theorem~\ref{th:MainA5D4}~\ref{it:MADa}}
\label{ss:ProofTh(i)}
Again, we restrict our attention to the case~I.

We set $\nu=\alpha_1+2\alpha_2+2\alpha_3+2\alpha_4+\alpha_5$.

We fix $(r,s)\in\mathbb N^2$. The weight of $b_{r,s}$ is $(r+2s)\nu$.
As in Section~\ref{ss:ProofProp}, a direct computation gives the
Lusztig data of $b_{r,s}$ w.r.t.\ the reduced word $\mathbf i$:
$$\mathbf n_{\mathbf i}(b_{r,s})=(r+s,r+s,0,0,0,s,s,0,0,0,r+s,r+s,0,0,0).$$

Let $\Omega$ be the orientation
$$
\begin{tikzpicture}
\node (1) at (1,0){$1$};
\node (2) at (2,0){$2$};
\node (3) at (3,0){$3$};
\node (4) at (4,0){$4$};
\node (5) at (5,0){$5$};
\draw[->] (1) to (2);
\draw[->] (3) to (2);
\draw[->] (3) to (4);
\draw[->] (5) to (4);
\end{tikzpicture}
$$
of the Dynkin diagram and let $\Lambda_\Omega$ be the preprojective
algebra it defines. A $\Lambda_\Omega$-module is the datum of a pair
$(\mathbf V,x)$, where $\mathbf V$ is an $I$-graded $\mathbf k$-vector
space and $x\in\Lambda_{\mathbf V,\Omega}$.

The category of $\Lambda_\Omega$-modules is tame and is described
in full details in~\cite{GeissLeclercSchroer05}. We refer to this
paper for more information. As usual, we denote by $S_i$ the
simple $\Lambda_\Omega$-module of dimension-vector $\alpha_i$.
Let $P$ and $Q_\lambda$ be the following $\Lambda_\Omega$-modules,
where $\lambda$ is a parameter in $\mathbf k\setminus\{0,1\}$.
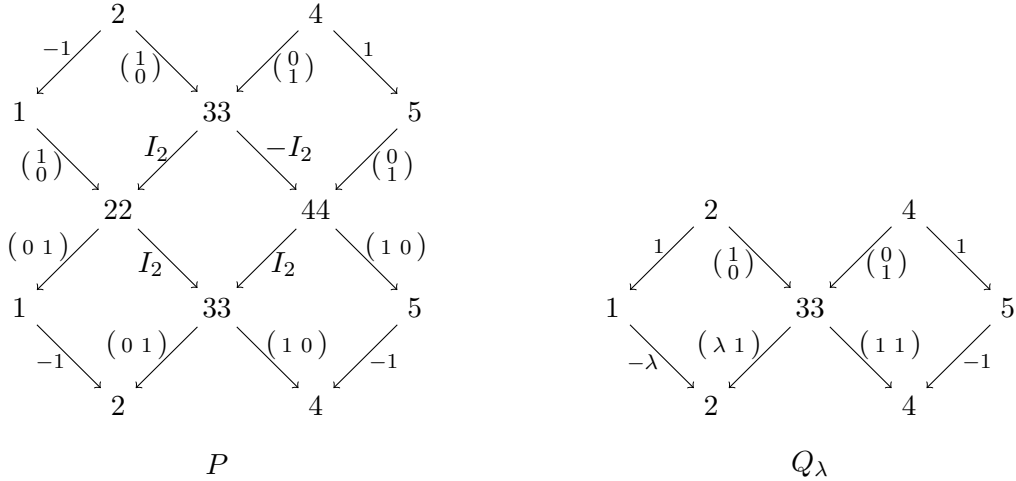
$$
\begin{tikzpicture}
\begin{scope}[scale=1.3]
\node (10) at (1,0){$2$};
\node (30) at (3,0){$4$};
\node (01) at (0,1){$1$};
\node (21) at (2,1){$33$};
\node (41) at (4,1){$5$};
\node (12) at (1,2){$22$};
\node (32) at (3,2){$44$};
\node (03) at (0,3){$1$};
\node (23) at (2,3){$33$};
\node (43) at (4,3){$5$};
\node (14) at (1,4){$2$};
\node (34) at (3,4){$4$};
\draw[->] (14) to node[pos=.3,left]{$\scriptstyle-1$} (03);
\draw[->] (14) to node[pos=.6,left]{$\left(\bsm1\\0\esm\right)$} (23);
\draw[->] (34) to node[pos=.6,right]{$\left(\bsm0\\1\esm\right)$} (23);
\draw[->] (34) to node[pos=.3,right]{$\scriptstyle1$} (43);
\draw[->] (03) to node[pos=.6,left]{$\left(\bsm1\\0\esm\right)$} (12);
\draw[->] (23) to node[pos=.3,left]{$I_2$} (12);
\draw[->] (23) to node[pos=.3,right]{$-I_2$} (32);
\draw[->] (43) to node[pos=.6,right]{$\left(\bsm0\\1\esm\right)$} (32);
\draw[->] (12) to node[pos=.3,left]{$\left(\bsm0&1\esm\right)$} (01);
\draw[->] (12) to node[pos=.6,left]{$I_2$} (21);
\draw[->] (32) to node[pos=.6,right]{$I_2$} (21);
\draw[->] (32) to node[pos=.3,right]{$\left(\bsm1&0\esm\right)$} (41);
\draw[->] (01) to node[pos=.6,left]{$\scriptstyle-1$} (10);
\draw[->] (21) to node[pos=.3,left]{$\left(\bsm0&1\esm\right)$} (10);
\draw[->] (21) to node[pos=.3,right]{$\left(\bsm1&0\esm\right)$} (30);
\draw[->] (41) to node[pos=.6,right]{$\scriptstyle-1$} (30);
\draw (2,-.6) node{$P$};
\end{scope}
\begin{scope}[scale=1.3,xshift=6cm]
\node (10) at (1,0){$2$};
\node (30) at (3,0){$4$};
\node (01) at (0,1){$1$};
\node (21) at (2,1){$33$};
\node (41) at (4,1){$5$};
\node (12) at (1,2){$2$};
\node (32) at (3,2){$4$};
\draw[->] (12) to node[pos=.3,left]{$\scriptstyle1$} (01);
\draw[->] (12) to node[pos=.6,left]{$\left(\bsm1\\0\esm\right)$} (21);
\draw[->] (32) to node[pos=.6,right]{$\left(\bsm0\\1\esm\right)$} (21);
\draw[->] (32) to node[pos=.3,right]{$\scriptstyle1$} (41);
\draw[->] (01) to node[pos=.6,left]{$\scriptstyle-\lambda$} (10);
\draw[->] (21) to node[pos=.3,left]{$\left(\bsm\lambda&1\esm\right)$} (10);
\draw[->] (21) to node[pos=.3,right]{$\left(\bsm1&1\esm\right)$} (30);
\draw[->] (41) to node[pos=.6,right]{$\scriptstyle-1$} (30);
\draw (2,-.6) node{$Q_\lambda$};
\end{scope}
\end{tikzpicture}
$$
In these pictures, as customary, a digit $i$ denotes a basis vector
in the vector space $\mathbf V_i$, and the small matrices that
adorn the arrows indicate the action of the arrows $h$.
The module $P$ is projective and injective. Each module
$Q_\lambda$ lies at the mouth of an homogeneous tube,
whence $\Ext^1(Q_\lambda,Q_\mu)=0$ for $\lambda\neq\mu$.

We fix an $I$-graded $\mathbf k$-vector space $\mathbf W$ of
dimension-vector $\nu$. Since the module $P$ is projective,
it is rigid: the closure of the orbit $\{x\in\Lambda_{\mathbf
W^2,\Omega}\mid(\mathbf W^2,x)\cong P\}$ is an irreducible component
of $\Lambda_{\mathbf W^2,\Omega}$. The orbit of each module
$Q_\lambda$ has codimension $1$ in $\Lambda_{\mathbf W,\Omega}$;
therefore, the closure of the union of these orbits is an irreducible
component of $\Lambda_{\mathbf W,\Omega}$.

Let $(r,s)\in\mathbb N^2$. By Section~2.6 in
\cite{GeissLeclercSchroer05}, the closure of the set of all
points $x\in\Lambda_{\mathbf W^{r+2s},\Omega}$ such that
$(\mathbf W^{r+2s},x)$ is isomorphic to a module of the form
$$M_{r,s}=Q_{\lambda_1}\oplus\cdots\oplus
Q_{\lambda_r}\oplus\underbrace{P\oplus\cdots\oplus
P}_{s\text{ times}}$$
is an irreducible component $\Lambda_{b,\Omega}$ of
$\Lambda_{\mathbf W^{r+2s},\Omega}$.

Observing that the word $\mathbf i$ is adapted to the orientation
$\Omega$, we can determine the Lusztig datum of $b$ in direction
$\mathbf i$ by looking at the multiplicities in a Krull-Remak-Schmidt
decomposition of the restriction of $M_{r,s}$ to the quiver
$(I,\Omega)$ (see Section~7.3 in \cite{BaumannKamnitzer12} for the
full justification of this procedure). We find
$$\mathbf n_{\mathbf i}(b)=(r+s,r+s,0,0,0,s,s,0,0,0,r+s,r+s,0,0,0)
=\mathbf n_{\mathbf i}(b_{r,s}),$$
from where we conclude that $b=b_{r,s}$.

Using the definition of the crystal structure on irreducible
components of the nilpotent varieties, one then deduces that
$$b_{r,s}=\sigma(b_{r,s})\quad\text{and}\quad
\varphi_i(b_{r,s})=
\begin{cases}
r+s&\text{if $i\in\{2,4\}$,}\\
0&\text{if $i\in\{1,3,5\}$.}
\end{cases}$$

Similarly, for $(t_1,t_3,t_5)\in\mathbb N^3$ and
$b=\tilde e_1^{t_1}\tilde e_3^{t_3}\tilde e_5^{t_5}b_{r,s}$,
a general point in $\Lambda_{b,\Omega}$ is the datum of the
structural maps of a $\Lambda_\Omega$-module isomorphic to
something of the form
$$Q_{\lambda_1}\oplus\cdots\oplus Q_{\lambda_r}\oplus
P^{\oplus s}\oplus S_1^{\oplus t_1}\oplus
S_3^{\oplus t_3}\oplus S_5^{\oplus t_5},$$
whence
$$\tilde e_1^{t_1}\tilde e_3^{t_3}\tilde e_5^{t_5}b_{r,s}=\sigma
\bigl(\tilde e_1^{t_1}\tilde e_3^{t_3}\tilde e_5^{t_5}b_{r,s}\bigr)
\quad\text{and}\quad
\varphi_i\bigl(\tilde e_1^{t_1}\tilde e_3^{t_3}\tilde
e_5^{t_5}b_{r,s}\bigr)=
\begin{cases}
r+s&\text{if $i\in\{2,4\}$,}\\
t_i&\text{if $i\in\{1,3,5\}$.}
\end{cases}$$
Given $(r',s',r'',s'')\in\mathbb N^4$, the equivalence
$$b_{r',s'}\leq_{\str}b_{r'',s''}\ \Longleftrightarrow\
\bigl(r'+2s'=r''+2s''\;\text{ and }\;r'\leq r''\bigr)$$
now follows from the definition of $\leq_{\str}$.

If $b_{r',s'}\leq_{\pol}b_{r'',s''}$, then $b_{r',s'}$ and
$b_{r'',s''}$ have the same weight, whence $r'+2s'=r''+2s''$,
and $\varphi_i(b_{r',s'})\leq\varphi_i(b_{r'',s''})$ for
all $i$, whence $r'\leq r''$. The converse implication
$$\bigl(r'+2s'=r''+2s''\;\text{ and }\;r'\leq r''\bigr)\
\Longrightarrow\ b_{r',s'}\leq_{\pol}b_{r'',s''}$$
follows from the two following facts:
first, $\Pol(b_{0,1})\subset\Pol(b_{2,0})$;
second, for any $(r,s)\in\mathbb N^2$,
$\Pol(b_{r,s})=r\Pol(b_{1,0})+s\Pol(b_{0,1})$,
where the sum is the Minkowski sum of convex bodies.
The first fact can be shown either by a direct computation,
or as a consequence of Theorem~\ref{th:CompCanSemican},
once noticed that
$$\langle S(b_{2,0})^*,G(b_{0,1})\rangle=
\langle S(b_{2,0})^*,\eta_1\rangle-
\langle S(b_{2,0})^*,G(b_{2,0})\rangle=
\binom21-1\neq0.$$
The second fact follows from the construction of
$\Pol(b)$ given in~\cite{BaumannKamnitzer12} or from
Remark~3.5~(ii) in~\cite{BaumannKamnitzerTingley11}.

\subsection{Proof of Theorem~\ref{th:MainA5D4}~\ref{it:MADc}}
\label{ss:ProofTh(iii)}
We keep the notation of the previous section.
We label the oriented edges in $H$ as follows.
$$
\begin{tikzpicture}[scale=1.6]
\node (1) at (1,0){$1$};
\node (2) at (2,0){$2$};
\node (3) at (3,0){$3$};
\node (4) at (4,0){$4$};
\node (5) at (5,0){$5$};
\draw[->] (1) to[out=20,in=160]
  node[midway,above]{$\scriptstyle h_1$} (2);
\draw[->] (3) to[out=160,in=20]
  node[midway,above]{$\scriptstyle h_2$} (2);
\draw[->] (3) to[out=20,in=160]
  node[midway,above]{$\scriptstyle h_3$} (4);
\draw[->] (5) to[out=160,in=20]
  node[midway,above]{$\scriptstyle h_4$} (4);
\draw[->] (2) to[out=-160,in=-20]
  node[midway,below]{$\scriptstyle\overline h_1$} (1);
\draw[->] (2) to[out=-20,in=-160]
  node[midway,below]{$\scriptstyle\overline h_2$} (3);
\draw[->] (4) to[out=-160,in=-20]
  node[midway,below]{$\scriptstyle\overline h_3$} (3);
\draw[->] (4) to[out=-20,in=-160]
  node[midway,below]{$\scriptstyle\overline h_4$} (5);
\end{tikzpicture}
$$
From the datum $(\mathbf V,x)$ of a $\Lambda_\Omega$-module, we will
define four maps
\begin{xalignat*}2
y&:\mathbf V_1\oplus\mathbf V_5\xrightarrow{x_{h_1}\oplus x_{h_4}}
\mathbf V_2\oplus\mathbf V_4,&
z&:\mathbf V_3\xrightarrow{\left(\bsm x_{h_2}\\x_{h_3}\esm\right)}
\mathbf V_2\oplus\mathbf V_4,\\[4pt]
\overline y&:\mathbf V_2\oplus\mathbf V_4\xrightarrow{x_{\overline
h_1}\oplus x_{\overline h_4}}\mathbf V_1\oplus\mathbf V_5,&
\overline z&:\mathbf V_2\oplus\mathbf V_4\xrightarrow{(\bsm
x_{\overline h_2}&x_{\overline h_3}\esm)}\mathbf V_3.
\end{xalignat*}

In the case of $P$, both maps $y$ and $z$ are injective and both
maps $\overline y$ and $\overline z$ are surjective. The subspace
$\im y\cap\im z$ has dimension-vector $\alpha_2+\alpha_4$ (it is
linearly spanned by the $2$ and the $4$ on the fifth line of the
picture of $P$). The subspace $\ker\overline y+\ker\overline z$
has dimension-vector $3(\alpha_2+\alpha_4)$ (it is spanned by the
$2$ and the $4$ on the third and fifth lines of the picture of
$P$). In addition, $(y\overline y)(\ker\overline y+\ker\overline
z)=\im y\cap\im z$ and $(y\overline y)^{-1}(\im y\cap\im z)=
\ker\overline y+\ker\overline z$.

In the case of $Q_\lambda$, both maps $y$ and $z$ are injective and
both maps $\overline y$ and $\overline z$ are surjective. We
have $\im y\cap\im z=\ker\overline y+\ker\overline z$; this subspace
has dimension-vector $\alpha_2+\alpha_4$ (it is linearly spanned
by the $2$ and the $4$ on the third line of the picture of
$Q_\lambda$). This subspace is also the kernel as well as the image
of $y\overline y$.

Let $(r,s)\in\mathbb N^2$. Set $p=r+s$ and $\mathbf V=\mathbf W^p$.
Let $(\mathbf j,\mathbf a)\in S_{p\nu}$ be so that
$\eta_p=\Theta_{\mathbf j}^{(\mathbf a)}$.
A flag of type $(\mathbf j,\mathbf a)$ in $\mathbf V$ is
the datum of subspaces $\mathbf X_i\subseteq\mathbf V_i$
for $i\in\{1,3,5\}$ and of $2$-steps filtrations
$0\subseteq\mathbf X'_i\subseteq\mathbf X''_i\subseteq\mathbf V_i$
for $i\in\{2,4\}$ of suitable dimension, namely
$$\dim\mathbf X_1=\dim\mathbf X'_2=\dim\mathbf X'_4=\dim\mathbf
X_5=p,\quad\dim\mathbf X_3=2p,\quad\dim\mathbf X''_2=\dim\mathbf
X''_4=3p.$$

Let $x$ be a general point in $\Lambda_{b_{2r,s},\Omega}$.
Thus, $(\mathbf V,x)$ is isomorphic to a $\Lambda_\Omega$-module
of the form
$$Q_{\lambda_1}\oplus\cdots\oplus Q_{\lambda_{2r}}\oplus
\underbrace{P\oplus\cdots\oplus P}_{s\text{ times}}$$
where $\lambda_1$, \dots, $\lambda_{2r}$ are distinct.
In agreement with this decomposition, we write $\mathbf V$ as a
direct sum $\mathbf V^{(Q)}\oplus\mathbf V^{(P)}$.

We look for $x$-stable flags of type $(\mathbf j,\mathbf a)$ in
$\mathbf V$. Since the maps $y$ and $z$ are injective, we must have
$y(\mathbf X_1\oplus\mathbf X_5)=z(\mathbf X_3)=\mathbf X'_2\oplus
\mathbf X'_4$, for dimension reasons. Likewise, the
surjectivity of $\overline y$ and $\overline z$ implies that
$\overline y^{-1}(\mathbf X_1\oplus\mathbf X_5)=\overline
z^{-1}(\mathbf X_3)=\mathbf X''_2\oplus\mathbf X''_4$.
Noticing that $y\overline y$ must map $\mathbf X''_2\oplus\mathbf
X''_4$ to $\mathbf X'_2\oplus\mathbf X'_4$, we get
\begin{align*}
(y\overline y)(\ker\overline y+\ker\overline z)\subseteq
\mathbf X'_2&\oplus\mathbf X'_4\subseteq\im y\cap\im z,\\[4pt]
\ker\overline y+\ker\overline z\subseteq
\mathbf X''_2&\oplus\mathbf X''_4\subseteq
(y\overline y)^{-1}(\im y\cap\im z).
\end{align*}

These conditions imply that $\mathbf X'_2$, $\mathbf X''_2$,
$\mathbf X'_4$ and $\mathbf X''_4$ are compatible
with the decomposition $\mathbf V=\mathbf V^{(Q)}\oplus\mathbf
V^{(P)}$ and determine the intersections of these four subspaces
with $\mathbf V^{(P)}$. For our problem of finding the $x$-stable
flags of type $(\mathbf j,\mathbf a)$, we can thus neglect
$\mathbf V^{(P)}$. Consequently, we now focus on $\mathbf V^{(Q)}$
and implicitly restrict the maps $x_h$, $y$, $z$, $\overline y$
and $\overline z$ to this subspace. We also simplify the notation by
writing $\mathbf X'_2$ instead of $\mathbf X'_2\cap\mathbf V^{(Q)}$,
and similarly for $\mathbf X''_2$, $\mathbf X'_4$
and~$\mathbf X''_4$.

After this renaming, $\dim\mathbf X'_2=\dim\mathbf X'_4=r$
and $\dim\mathbf X''_2=\dim\mathbf X''_4=3r$. Let
$\mathbf Y_2\oplus\mathbf Y_4$ denote the subspace
$\im y\cap\im z$ in $\mathbf V^{(Q)}$. Then
$\dim\mathbf Y_2=\dim\mathbf Y_4=2r$. We must have
\begin{gather*}
\mathbf X'_2\subseteq\mathbf Y_2\subseteq\mathbf X''_2,\\[4pt]
\mathbf X'_4\subseteq \mathbf Y_4\subseteq\mathbf X''_4,\\[4pt]
(y\overline y)(\mathbf X''_2\oplus\mathbf X''_4)\subseteq
\mathbf X'_2\oplus\mathbf X'_4,\\[4pt]
(z\overline z)(\mathbf X''_2\oplus\mathbf X''_4)\subseteq
\mathbf X'_2\oplus\mathbf X'_4.
\end{gather*}

Recalling the preprojective relations
$x_{h_1}x_{\overline h_1}+x_{h_2}x_{\overline h_2}=0$ and
$x_{h_3}x_{\overline h_3}+x_{h_4}x_{\overline h_4}=0$,
we can rephrase this in terms of representations of a tame quiver:
$$\xymatrix@R=3.3em@C=4em{\mathbf X''_2/\mathbf Y_2\ar[r]\ar[dr]&
\mathbf X'_2\\\mathbf X''_4/\mathbf Y_4\ar[r]\ar[ur]&\mathbf X'_4}
\qquad\raisebox{-2.4em}{is a subrepresentation of}\qquad
\xymatrix@R=3.3em@C=4em{\mathbf V^{(Q)}_2/\mathbf Y_2
\ar[r]^(.55){x_{h_2}x_{\overline h_2}}
\ar[dr]_(.25){x_{h_3}x_{\overline h_2}}
&\mathbf Y_2\\
\mathbf V^{(Q)}_4/\mathbf Y_4
\ar[r]_(.55){x_{h_3}x_{\overline h_3}}
\ar[ur]_(.75){x_{h_2}x_{\overline h_3}}
&\mathbf Y_4.}$$

Choosing adequate bases in $\mathbf V^{(Q)}_2$ and
$\mathbf V^{(Q)}_4$, the linear maps on the right diagram are
given by the identity matrix of size $2r$, except for the top
horizontal arrow, which is represented by a diagonal matrix
with coefficients $\lambda_1$, \dots, $\lambda_{2r}$.
Subrepresentations of the required dimension are obtained by
taking $r$ among the $2r$ eigenspaces of the composed map
$$(x_{h_3}x_{\overline h_2})^{-1}(x_{h_3}x_{\overline h_3})
(x_{h_2}x_{\overline h_3})^{-1}(x_{h_2}x_{\overline h_2}).$$

All in all, we have $\binom{2r}r$ $x$-stable flags of type
$(\mathbf j,\mathbf a)$ in $\mathbf V$. The Euler
characteristic of this set of flags is thus equal to $\binom{2r}r$,
and also to $\kappa\bigl(\Theta_{\mathbf j}^{(\mathbf a)}\bigr)(x)
=\langle S(b_{2r,s})^*,\eta_p\rangle$ (see Section~\ref{ss:BackLuszCons}).

\subsection{Analysis of extensions of quiver representations}
\label{ss:AnalExtQuivRep}
Again, we consider the Dynkin diagram of type $A_5$. The reduced word
$$\mathbf i=(2,4,1,3,5,2,4,1,3,5,2,4,1,3,5)$$
is
adapted to the following orientation $\Omega$, in the sense
of \cite{Lusztig90}, Section~4.7.
$$
\begin{tikzpicture}
\node (1) at (1,0){$1$};
\node (2) at (2,0){$2$};
\node (3) at (3,0){$3$};
\node (4) at (4,0){$4$};
\node (5) at (5,0){$5$};
\draw[->] (1) to (2);
\draw[->] (3) to (2);
\draw[->] (3) to (4);
\draw[->] (5) to (4);
\end{tikzpicture}
$$

Our interest lies in the representation theory over $\mathbf k$ of the
quiver $Q=(I,\Omega)$. The dimension-vector map induces a bijection
from the set of isomorphism classes of indecomposable representations
of $Q$ onto the set of positive roots (Gabriel's theorem). For each
positive root $\beta$, we pick an indecomposable $\mathbf kQ$-module
$M(\beta)$ of dimension-vector $\beta$. These modules can be organized
in an Auslander-Reiten quiver, as follows.
$$
\begin{tikzpicture}[xscale=1.2]
\node (03) at (0,3){$M(\alpha_2)$};
\node (01) at (0,1){$M(\alpha_4)$};
\node (14) at (1,4){$M(\alpha_{1,2})$};
\node (12) at (1,2){$M(\alpha_{2,4})$};
\node (10) at (1,0){$M(\alpha_{4,5})$};
\node (23) at (2,3){$M(\alpha_{1,4})$};
\node (21) at (2,1){$M(\alpha_{2,5})$};
\node (34) at (3,4){$M(\alpha_{3,4})$};
\node (32) at (3,2){$M(\alpha_{1,5})$};
\node (30) at (3,0){$M(\alpha_{2,3})$};
\node (43) at (4,3){$M(\alpha_{3,5})$};
\node (41) at (4,1){$M(\alpha_{1,3})$};
\node (54) at (5,4){$M(\alpha_5)$};
\node (52) at (5,2){$M(\alpha_3)$};
\node (50) at (5,0){$M(\alpha_1)$};
\draw[->] (03) to (14);
\draw[->] (03) to (12);
\draw[->] (01) to (12);
\draw[->] (01) to (10);
\draw[->] (14) to (23);
\draw[->] (12) to (23);
\draw[->] (12) to (21);
\draw[->] (10) to (21);
\draw[->] (23) to (34);
\draw[->] (23) to (32);
\draw[->] (21) to (32);
\draw[->] (21) to (30);
\draw[->] (34) to (43);
\draw[->] (32) to (43);
\draw[->] (32) to (41);
\draw[->] (30) to (41);
\draw[->] (43) to (54);
\draw[->] (43) to (52);
\draw[->] (41) to (52);
\draw[->] (41) to (50);
\end{tikzpicture}
$$
Here, we wrote $\alpha_{i,j}$ for $\alpha_i+\alpha_{i+1}+\cdots+\alpha_j$.
The enumeration $(\beta_1,\ldots,\beta_{15})$ of the positive
roots defined by $\mathbf i$, namely
$$(\,\alpha_2,\ \alpha_4,\
\alpha_{1,2},\ \alpha_{2,4},\ \alpha_{4,5},\
\alpha_{1,4},\ \alpha_{2,5},\
\alpha_{3,4},\ \alpha_{1,5},\ \alpha_{2,3},\
\alpha_{3,5},\ \alpha_{1,3},\
\alpha_5,\ \alpha_3,\ \alpha_1\,),$$
can be read from this Auslander-Reiten quiver, proceeding columnwise
from the top left corner to the bottom right corner. The simple
$\mathbf kQ$-modules are of course the modules $M(\alpha_i)$,
for $i\in I$; as customary, we denote them by $S_i$.

Consider the following lists of $\mathbf kQ$-modules.
\begin{center}
\setlength\tabcolsep{1.4em}
\begin{tabular}{>{$}l<{$}>{$}l<{$}>{$}l<{$}}
\toprule
t&M_t&N_t\\
\midrule
1&M(\alpha_2)&0\\[3pt]
2&M(\alpha_{1,2})&M(\alpha_1)\\[3pt]
3&M(\alpha_{1,2})\oplus M(\alpha_{2,4})&M(\alpha_{1,4})\\[3pt]
4&M(\alpha_{1,2})\oplus M(\alpha_{2,5})&M(\alpha_{1,5})\\[3pt]
5&M(\alpha_{1,2})\oplus M(\alpha_{2,3})&M(\alpha_{1,3})\\[3pt]
6&M(\alpha_{2,4})&M(\alpha_{3,4})\\[3pt]
7&M(\alpha_{1,4})&M(\alpha_1)\oplus M(\alpha_{3,4})\\[3pt]
8&M(\alpha_{1,4})\oplus M(\alpha_{2,5})&
M(\alpha_{1,5})\oplus M(\alpha_{3,4})\\[3pt]
9&M(\alpha_{1,4})\oplus M(\alpha_{2,3})&
M(\alpha_{1,3})\oplus M(\alpha_{3,4})\\[3pt]
10&M(\alpha_{2,5})&M(\alpha_{3,5})\\[3pt]
11&M(\alpha_{1,5})&M(\alpha_1)\oplus M(\alpha_{3,5})\\[3pt]
12&M(\alpha_{1,5})\oplus M(\alpha_{2,3})&
M(\alpha_{1,3})\oplus M(\alpha_{3,5})\\[3pt]
13&M(\alpha_{2,3})&M(\alpha_3)\\[3pt]
14&M(\alpha_{1,3})&M(\alpha_1)\oplus M(\alpha_3)\\[3pt]
\bottomrule
\end{tabular}
\end{center}

\begin{proposition}
\label{pr:ListExtS2}
Let $0\to S_2\xrightarrow fM\to N\to0$ be a short exact sequence of
$\mathbf kQ$-modules. Then there exists $t\in\{1,\ldots,14\}$
and a $\mathbf k\mathbf Q$-module $L$ such that $M\cong M_t\oplus L$
and $N\cong N_t\oplus L$.
\end{proposition}
\begin{proof}
Let $K=\{1,3,4,6,7,9,10,12\}$. In the $\mathbf kQ$-module $M(\beta_k)$,
the vector space attached to the vertex $2$ is nonzero if and only
if $k\in K$. In this case, it is one dimensional, spanned by a vector
$e_k$.

We endow $K$ with an order by saying that $k\succ\ell$ is there
is a path of positive length from $M(\beta_k)$ to $M(\beta_\ell)$
in the Auslander-Reiten quiver. Thus for instance
$\alpha_{2,4}\succ\alpha_{1,3}$ but
$\alpha_{1,2}\not\succ\alpha_{2,3}$.

Writing $M$ as a direct sum of indecomposable modules, we find an
isomorphism
$$g:M\xrightarrow\simeq\bigoplus_{k=1}^{15}M(\beta_k)\otimes W_k,$$
where the vector spaces $W_k$ account for the multiplicities. Then
$(g\circ f)(e_1)$ has the form $\sum_{k\in K}e_k\otimes w_k$,
where $w_k\in W_k$.

Using the explicit form of the modules $M(\beta_k)$, one sees by a
case-by-case analysis that it is possible to modify $g$ so as to
cancel the elements $w_\ell$ each time there is a $k\succ\ell$
such that $w_k\neq0$. After this modification, the elements in the
set $K'=\{k\in K\mid w_k\neq0\}$ are pairwise incomparable. This
leaves a list of fourteen possibilities for $K'$ (note that $K'$
cannot be empty).

For $k\in K'$, choose a complementary subspace $W'_k$ in $W_k$ to
the line $\mathbf kw_k$, and for $k\notin K'$, set $W'_k=W_k$. Let
$$L=\bigoplus_{k=1}^{15}M(\beta_k)\otimes W'_k.$$
Then $M\cong M_t\oplus L$ for a certain $t\in\{1,\ldots,14\}$.
In this isomorphism, the image of $f$ is contained in $M_t$, so
the cokernel $N$ of $f$ is isomorphic to $N_t\oplus L$.
\end{proof}

One can of course analyze in a similar fashion the extensions
by $S_4$, using for instance the diagram automorphism.

To a $\mathbf kQ$-module $M$, we associate the row vector
$\mathbf n(M)=(n_1,\ldots,n_{15})$, where $n_k$ is the
multiplicity of $M(\beta_k)$ in a Krull-Remak-Schmidt
decomposition of $M$. Then $M\mapsto\mathbf n(M)$ induces
a bijection from the set of isomorphism classes of
$\mathbf kQ$-modules onto $\mathbb N^{15}$.

Consider the following elements in $\mathbb N^{15}$.

\vspace{-12pt}
\begin{center}
\begin{tabular}{>{$}p{1.8em}<{$\hfill}@{$\ =\ ($}
>{\hfill$}p{1em}<{$}@{$,\;$}
*{13}{>{\hfill$}p{1.4em}<{$}@{$,\;$}}
>{\hfill$}p{1.4em}<{$}@{$\;)$}}
\mathbf w_1&1&0&0&0&0&0&0&0&0&0&0&-1&0&1&1\\[1pt]
\mathbf w_2&0&1&0&0&0&0&0&0&0&0&-1&0&1&1&0\\[1pt]
\mathbf w_3&0&0&1&0&0&0&0&0&0&0&0&-1&0&1&0\\[1pt]
\mathbf w_4&0&0&1&0&0&0&1&0&-1&0&0&-1&0&1&1\\[1pt]
\mathbf w_5&0&0&1&1&0&-1&0&0&0&0&0&-1&0&1&1\\[1pt]
\mathbf w_6&0&0&0&1&0&0&0&-1&0&0&0&-1&0&1&1\\[1pt]
\mathbf w_7&0&0&0&0&0&1&1&-1&-1&0&0&-1&0&1&1\\[1pt]
\mathbf w_8&0&0&0&0&0&1&1&0&-1&-1&-1&0&1&1&0\\[1pt]
\mathbf w_9&0&0&0&1&0&0&0&0&0&-1&-1&0&1&1&0\\[1pt]
\mathbf w_{10}&0&0&0&1&1&0&-1&0&0&0&-1&0&1&1&0\\[1pt]
\mathbf w_{11}&0&0&0&0&1&1&0&0&-1&0&-1&0&1&1&0\\[1pt]
\mathbf w_{12}&0&0&0&0&1&0&0&0&0&0&-1&0&0&1&0\\[1pt]
\mathbf w_{13}&0&0&0&0&0&0&0&1&0&0&-1&0&1&0&0\\[1pt]
\mathbf w_{14}&0&0&0&0&0&0&0&0&1&0&-1&-1&0&1&0\\[1pt]
\mathbf w_{15}&0&0&0&0&0&0&0&0&0&1&0&-1&0&0&1\\[1pt]
\mathbf w_{16}&0&0&0&0&0&1&0&-1&0&0&0&-1&0&1&0\\[1pt]
\mathbf w_{17}&0&0&0&0&0&0&1&0&0&-1&-1&0&0&1&0\\[1pt]
\mathbf x&0&0&0&0&0&0&0&0&0&0&1&1&-1&-2&-1\\[1pt]
\mathbf x'&0&0&0&0&0&0&0&0&0&0&0&1&0&-1&-1\\[1pt]
\mathbf y&1&1&0&0&0&0&0&0&0&0&1&1&0&0&0\\[1pt]
\mathbf z&1&1&0&0&0&1&1&0&0&0&1&1&0&0&0
\end{tabular}
\end{center}

\begin{corollary}
\label{co:ExtS2&S4}
Let $p\in\mathbb N$ and let $M$ et $N$ be two $\mathbf kQ$-modules.
If there is a short exact sequence
$$0\to S_2^{\oplus p}\oplus S_4^{\oplus p}\to M\to N\to0,$$
then $\mathbf n(M)-\mathbf n(N)-p\mathbf x$ belongs to the
$\mathbb N$-span of the vectors $\mathbf w_k$.
\end{corollary}
\begin{proof}
With the notation of Proposition~\ref{pr:ListExtS2}, it suffices
to check that each $\mathbf n(M_t)-\mathbf n(N_t)-\mathbf x'$
can be written as a (possibly empty) sum of vectors~$\mathbf w_k$.
\end{proof}

\subsection{Proof of Theorem~\ref{th:MainA5D4}~\ref{it:MADb}}
\label{ss:ProofTh(ii)}
As before, we focus on the case~I.

Let $\nu=\alpha_1+2\alpha_2+2\alpha_3+2\alpha_4+\alpha_5$ and
let $p\in\mathbb N$. Define $(\mathbf j,\mathbf a)\in S_{p\nu}$
so that $\eta_p=\Theta_{\mathbf j}^{(\mathbf a)}$. Let $\mathbf V$
be an $I$-graded $\mathbf k$-vector space of dimension-vector $2p\nu$
and recall the notation set up in Section~\ref{ss:BackLuszCons}.
To show Theorem~\ref{th:MainA5D4}~\ref{it:MADb}, it suffices to
prove the equation
\begin{equation}
\label{eq:DecompLja}
L_{\mathbf j,\mathbf a;\Omega}=L_{b_{0,p},\Omega}\oplus
L_{b_{2,p-1},\Omega}\oplus L_{b_{4,p-2},\Omega}\oplus\cdots\oplus
L_{b_{2p,0},\Omega}.
\end{equation}
To prove~\eqref{eq:DecompLja}, we will show that the map
$\pi_{\mathbf j,\mathbf a}$ is semismall with respect to the
stratification given by the $G_{\mathbf V}$-orbits on
$\mathbf E_{\mathbf V,\Omega}$ and identify the relevant strata.
The precise statement is given in Proposition~\ref{pr:SemiSmallness}
below.

We regard a flag of type $(\mathbf j,\mathbf a)$ in $\mathbf V$ as
a $5$-step filtration
\begin{equation}
\label{eq:5StepFilt}
\mathbf V=\mathbf V^0\supseteq\mathbf V^1\supseteq\mathbf V^2
\supseteq\mathbf V^3\supseteq\mathbf V^4\supseteq\mathbf V^5=0,
\end{equation}
such that the dimension-vector of $\mathbf V^{k-1}/\mathbf V^k$ is
$p(\alpha_2+\alpha_4)$, $p(\alpha_1+2\alpha_3+\alpha_5)$, or
$2p(\alpha_2+\alpha_4)$ according to whether $k\in\{1,5\}$,
$k\in\{2,4\}$, or $k=3$.

To a row vector $\mathbf n=(n_1,\ldots,n_{15})$ in
$\mathbb N^{15}$, we associate the weight
$|\mathbf n|=n_1\beta_1+\cdots+n_{15}\beta_{15}$.
This is the dimension-vector of any $\mathbf kQ$-module $M$
such that $\mathbf n(M)=\mathbf n$.

A point $(x,\mathbf V^\bullet)$ in
$\widetilde{\mathcal F}_{\mathbf j,\mathbf a}$ yields
$\mathbf kQ$-modules $M_{k,\ell}=(\mathbf V^k/\mathbf V^\ell,x)$,
for all $k<\ell$. Given $\mathbf u$ and $\mathbf v$ in
$\mathbb N^{15}$ such that $|\mathbf u|=2p\nu$ and
$|\mathbf v|=p\nu$, we denote by
$\widetilde{\mathcal F}^{\mathbf u,\mathbf v}$ the set of all
$(x,\mathbf V^\bullet)\in\widetilde{\mathcal F}_{\mathbf j,\mathbf a}$
such that $\mathbf n(M_{1,3})=\mathbf v$ and
$\mathbf n(M_{0,5})=\mathbf u$.

Likewise, for $\mathbf u\in\mathbb N^{15}$ such that
$|\mathbf u|=2p\nu$, we denote by $\mathscr O_{\mathbf u}$
the set of all $x\in\mathbf E_{\mathbf V,\Omega}$ such that
$\mathbf n((\mathbf V,x))=\mathbf u$. This is a
$G_{\mathbf V}$-orbit in $\mathbf E_{\mathbf V,\Omega}$.
If $b\in B(-\infty)$ has Lusztig datum $\mathbf u$ in direction
$\mathbf i$, then $L_{b,\Omega}$ is the intersection cohomology
sheaf on $\overline{\mathscr O_{\mathbf u}}$.

In this fashion, we partition
$\widetilde{\mathcal F}_{\mathbf j,\mathbf a}$ and
$\mathbf E_{\mathbf V,\Omega}$ into locally closed pieces. The
first projection $\pi_{\mathbf j,\mathbf a}:\widetilde{\mathcal F}
\to\mathbf E_{\mathbf V,\Omega}$ restricts to a map
$\pi^{\mathbf u,\mathbf v}:\widetilde{\mathcal F}^{\mathbf
u,\mathbf v}\to\mathscr O_{\mathbf u}$.

\begin{proposition}
\label{pr:SemiSmallness}
Assume that
$\widetilde{\mathcal F}^{\mathbf u,\mathbf v}\neq\varnothing$.
Then the map $\pi^{\mathbf u,\mathbf v}$ is a locally trivial
fibration with a smooth and connected fiber. Moreover,
for $x\in\mathscr O_{\mathbf u}$,
\begin{equation}
\label{eq:SemiSmallness}
2\dim(\pi^{\mathbf u,\mathbf v})^{-1}(x)+\dim\mathscr O_{\mathbf u}
\leq\dim\widetilde{\mathcal F}_{\mathbf j,\mathbf a},
\end{equation}
with equality if and only if $\mathbf u=2(p-s)\mathbf y+s\mathbf z$
and $\mathbf v=\mathbf u-p\mathbf y$ for some $s\in\{0,\ldots,p\}$.
\end{proposition}

The Lusztig datum in direction $\mathbf i$ of the element $b_{r,s}$
is $r\mathbf y+s\mathbf z$ (see Section~\ref{ss:ProofTh(i)}), so
for $\mathbf u=2(p-s)\mathbf y+s\mathbf z$, the intersection
cohomology sheaf on $\overline{\mathscr O_{\mathbf u}}$ with
coefficients in the trivial local system is $L_{b_{2(p-s),s},\Omega}$.
In view of Proposition~\ref{pr:SemiSmallness}, and since each stratum
$\mathscr O_{\mathbf u}$ is simply connected, the Decomposition
Theorem for semismall maps then implies equation~\eqref{eq:DecompLja}.

The remainder of this section is devoted to the proof of
Proposition~\ref{pr:SemiSmallness}. We have to compute the
difference $\Delta$ between the right and left-hand sides of
\eqref{eq:SemiSmallness} and to show that $\Delta\geq0$.
The assertion about the smoothness and the connectedness of the
fiber of $\pi^{\mathbf u,\mathbf v}$ will be proved on the way.

Lemma~1.6~(c) in~\cite{Lusztig91} gives
$\dim\widetilde{\mathcal F}_{\mathbf j,\mathbf a}=40p^2$.
In addition, $\dim\mathbf E_{\mathbf V,\Omega}=48p^2$.

The formulas \eqref{eq:RingelFormulas} (see Section~\ref{ss:DegOrder})
give a recipe to compute the matrices $H$ and $E$ with entries
\begin{align*}
h_{k,\ell}&=\dim\Hom_{\mathbf kQ}(M(\beta_k),M(\beta_\ell)),\\[2pt]
e_{k,\ell}&=\dim\Ext^1_{\mathbf kQ}(M(\beta_k),M(\beta_\ell)).
\end{align*}

We can then find the dimension of any $G_{\mathbf V}$-orbit
contained in $\mathbf E_{\mathbf V,\Omega}$: by Lemma~1
in~\cite{Crawley-Boevey92}, \S3, we have
$$\dim\mathbf E_{\mathbf V,\Omega}-\dim\mathscr O_{\mathbf u}=
\mathbf u\,E\,\mathbf u^T,$$
where the superscript $T$ denotes the matrix transposition.

Now we fix $\mathbf u$ and $\mathbf v$ in $\mathbb N^{15}$ such
that $|\mathbf u|=2p\nu$ and $|\mathbf v|=p\nu$ and we fix
$x\in\mathscr O_{\mathbf u}$.

If there exists a flag $\mathbf V^\bullet$ such that
$(x,\mathbf V^\bullet)\in\widetilde{\mathcal F}^{\mathbf u,\mathbf v}$,
then we can consider the $\mathbf kQ$-modules
$M_{k,\ell}=(\mathbf V^k/\mathbf V^\ell,x)$. They are related by
$$M_{0,5}\cong M_{1,5}\oplus\bigl(S_2^{\oplus p}\oplus S_4^{\oplus p}\bigr)
\quad\text{and}\quad
M_{0,4}\cong M_{1,3}\oplus\bigl(S_1^{\oplus p}\oplus S_2^{\oplus p}
\oplus S_3^{\oplus2p}\oplus S_4^{\oplus p}\oplus S_5^{\oplus p}\bigr),$$
because $S_2$ and $S_4$ are projective and $S_1$, $S_3$ and $S_5$ are
injective. Setting
\vspace{-8pt}
\begin{center}
\begin{tabular}{>{$}p{1em}<{$\hfill}@{$\ =\ ($}
>{\hfill$}p{.8em}<{$}@{$,\;$}
*{13}{>{\hfill$}p{.8em}<{$}@{$,\;$}}
>{\hfill$}p{.8em}<{$}@{$\;)$,}}
\mathbf s&1&1&0&0&0&0&0&0&0&0&0&0&0&0&0\\[1pt]
\mathbf t&1&1&0&0&0&0&0&0&0&0&0&0&1&2&1
\end{tabular}
\end{center}
\vspace{-8pt}
we thus have $\mathbf n(M_{1,5})=\mathbf u-p\mathbf s$ and
$\mathbf n(M_{0,4})=\mathbf v+p\mathbf t$.
Further, Corollary~\ref{co:ExtS2&S4} asserts that
$\mathbf u-\mathbf n(M_{0,4})-p\mathbf x$ belongs to the
$\mathbb N$-span of the vectors $\mathbf w_k$. Denoting by
$W$ the matrix whose lines are the vectors $\mathbf w_k$, we
deduce the existence of a row vector
$\boldsymbol\tau=(\tau_1,\ldots,\tau_{17})$ in $\mathbb N^{17}$
such that $\mathbf u-\mathbf v-p\mathbf y=\boldsymbol\tau\,W$.

The choice of a flag $\mathbf V^\bullet$ in the fiber
$(\pi^{\mathbf u,\mathbf v})^{-1}(x)$ can be decomposed in three
steps: first the choice of $\mathbf V^1$, then the choice
of $\mathbf V^3$, and finally the choice of $\mathbf V^2$ and
$\mathbf V^4$.

We thus begin with $\mathbf V^1$. The vector space $\mathbf V^1_2$
has codimension $p$ in $\mathbf V_2$ and contains the image of
$x_{h_1}$ and $x_{h_2}$. To choose it therefore amounts to choose
a codimension $p$ subspace in a space of dimension $u_1$; in other
words, to a point in a Grassmannian of dimension $p(u_1-p)$.
Likewise the choice of $\mathbf V^1_4$ amounts to the choice of
a point in a Grassmannian of dimension $p(u_2-p)$.
(If $p$ is larger than $u_1$ or $u_2$, then the fiber
$(\pi^{\mathbf u,\mathbf v})^{-1}(x)$ is empty.) The choice
of $\mathbf V^1$ therefore contributes a smooth connected variety
of dimension $p(u_1+u_2)-2p^2$ to the fiber.

Now suppose that $\mathbf V^1$ has been chosen, and pick
a $\mathbf kQ$-module $N$ such that $\mathbf n(N)=\mathbf v$.
An $x$-stable $I$-graded $\mathbf k$-vector subspace $\mathbf V^3$
defines a $\mathbf kQ$-module $M_{1,3}=(\mathbf V^1/\mathbf V^3,x)$.
The datum of a $\mathbf V^3$ such that $\mathbf n(M_{1,3})=\mathbf v$
is equivalent to the datum of a surjective morphism $M_{1,5}\to N$,
up to composition with an automorphism of $N$. Surjective morphisms
form an open dense subset in $\Hom_{\mathbf kQ}(M_{1,5},N)$
(or do not exist at all, in which case the fiber
$(\pi^{\mathbf u,\mathbf v})^{-1}(x)$ is empty). Therefore the choice
of $\mathbf V^3$ contributes a smooth connected variety of dimension
$$\dim\Hom_{\mathbf kQ}(M_{1,5},N)-\dim\Hom_{\mathbf kQ}(N,N)
=(\mathbf u-p\mathbf s)\,H\,\mathbf v^T-\mathbf v\,H\,\mathbf v^T$$
to the fiber.

Lastly, we notice that the choice of $\mathbf V^1$ and $\mathbf V^3$
fully determines $\mathbf V^2$ and $\mathbf V^4$:
$$\mathbf V^2_i=
\begin{cases}
\mathbf V^3_i&\text{if $i\in\{1,3,5\}$,}\\
\mathbf V^1_i&\text{if $i\in\{2,4\}$,}
\end{cases}
\qquad\qquad
\mathbf V^4_i=
\begin{cases}
0&\text{if $i\in\{1,3,5\}$,}\\
\mathbf V^3_i&\text{if $i\in\{2,4\}$.}
\end{cases}$$
In other words, the third step does not contribute further to the fiber.

Taking every contribution into account, we find that the difference
between the two sides of~\eqref{eq:SemiSmallness} is equal to
$$\Delta=-4p^2+\mathbf u\,E\,\mathbf u^T-2p\,\mathbf u\,\mathbf s^T
-2(\mathbf u-p\mathbf s-\mathbf v)\,H\,\mathbf v^T.$$
In this expression, we substitute
$\mathbf u=\mathbf v+p\mathbf y+\boldsymbol\tau\,W$. Then
$\Delta$ is a polynomial of degree $2$ in the variables $\mathbf v$,
$p$, and $\boldsymbol\tau$. Our aim is to show that $\Delta\geq0$
under the assumptions that the variables are nonnegative and that
$|\mathbf v|=p\nu$.

The equation $|\mathbf v|=p\nu$ is equivalent to the system
\begin{align*}
p-(v_3+v_6+v_9+v_{12}+v_{15})&=0,\\
2p-(v_1+v_3+v_4+v_6+v_7+v_9+v_{10}+v_{12})&=0,\\
2p-(v_4+v_6+v_7+v_8+v_9+v_{10}+v_{11}+v_{12}+v_{14})&=0,\\
2p-(v_2+v_4+v_5+v_6+v_7+v_8+v_9+v_{11})&=0,\\
p-(v_5+v_7+v_9+v_{11}+v_{13})&=0.
\end{align*}
We denote the left-hand sides of these equations by $L_1$, \dots,
$L_5$, from top to bottom.

Let us write $\Delta=\Delta_0+\Delta_1+\Delta_2$, where $\Delta_d$
collects the terms of degree $d$ in the variables $\tau_k$. Then
$$\Delta_0
=\mathbf v\,E\,\mathbf v^T+p\,\bigl[\mathbf y\,(E+E^T)
-2\mathbf s-2(\mathbf y-\mathbf s)\,H\bigr]\,\mathbf v^T
+p^2(-4+\mathbf y\,E\,\mathbf y^T-2\,\mathbf y\,\mathbf s^T).$$
If we add
$$L_1(L_1-L_2+v_{15})+L_2(L_2+v_1+v_{12})+L_3(L_3-L_2-L_4+v_{14})
+L_4(L_4+v_2+v_{11})+L_5(L_5-L_4+v_{13})$$
to $\Delta_0$, we get, after a lengthy but straightforward
calculation
\begin{align*}
&(p-v_1)(v_{12}-p)+(p-v_2)(v_{11}-p)\\
&+(v_3+v_4+v_6)(v_6+v_8+v_9)+(v_4+v_5+v_7)(v_7+v_9+v_{10})\\
&+(v_3^2-v_3v_8+v_8^2)+(v_4^2-v_4v_9+v_9^2)+(v_5^2-v_5v_{10}+v_{10}^2).
\end{align*}
The first term is
\begin{multline*}
\left(\frac{v_{12}-v_1}2\right)^2
-\left(p-\frac{v_1+v_{12}}2\right)^2\\
\equiv
\left(\frac{v_{12}-v_1}2\right)^2
-\left(\frac{v_3+v_4+v_6}2+\frac{v_7+v_9+v_{10}}2\right)^2
\pmod{L_2}
\end{multline*}
and the second term is
\begin{multline*}
\left(\frac{v_{11}-v_2}2\right)^2
-\left(p-\frac{v_2+v_{11}}2\right)^2\\
\equiv
\left(\frac{v_{11}-v_2}2\right)^2
-\left(\frac{v_4+v_5+v_7}2+\frac{v_6+v_8+v_9}2\right)^2
\pmod{L_4}.
\end{multline*}
Modulo the relations $L_i$, $\Delta_0$ is thus congruent to
\begin{align*}
\widetilde\Delta_0=
&\left(\frac{v_{12}-v_1}2\right)^2
+\left(\frac{v_{11}-v_2}2\right)^2\\
&-\left(\frac{v_3+v_4+v_6}2+\frac{v_7+v_9+v_{10}}2\right)^2
-\left(\frac{v_4+v_5+v_7}2+\frac{v_6+v_8+v_9}2\right)^2\\[2pt]
&+(v_3+v_4+v_6)(v_6+v_8+v_9)+(v_4+v_5+v_7)(v_7+v_9+v_{10})\\[2pt]
&+\left(\frac{v_3-v_8}2+\frac{v_4-v_9}2\right)^2
+\left(\frac{v_4-v_9}2+\frac{v_5-v_{10}}2\right)^2
+\left(\frac{v_3-v_8}2-\frac{v_5-v_{10}}2\right)^2\\
&+\left(\frac{v_3-v_8}2-\frac{v_4-v_9}2+\frac{v_5-v_{10}}2\right)^2
+\left(\frac{v_3+v_8}2\right)^2
+\left(\frac{v_4+v_9}2\right)^2
+\left(\frac{v_5+v_{10}}2\right)^2.
\end{align*}
A final transformation yields
\begin{align*}
\widetilde\Delta_0=
&\left(\frac{v_{12}-v_1}2\right)^2
+\left(\frac{v_{11}-v_2}2\right)^2\\
&+\frac12\left(v_6-v_7+\frac{v_3-v_5+v_8-v_{10}}2\right)^2
+\frac12\left(\frac{v_3-v_8}2-\frac{v_5-v_{10}}2\right)^2\\
&+\left(\frac{v_3-v_8}2-\frac{v_4-v_9}2+\frac{v_5-v_{10}}2\right)^2
+\left(\frac{v_3+v_8}2\right)^2
+\left(\frac{v_4+v_9}2\right)^2
+\left(\frac{v_5+v_{10}}2\right)^2.
\end{align*}

We now turn to
$$\Delta_1=\boldsymbol\tau\,W(E+E^T-2H)\,\mathbf v^T+
p\,\boldsymbol\tau\,W\bigl[(E+E^T)\,\mathbf y^T-2\mathbf s^T\bigr].$$
We add
\begin{align*}
&\tau_1(L_1-2L_2+L_3)+\tau_2(L_3-2L_4+L_5)
-\tau_3L_1-(\tau_4+\tau_5)L_2\\
&-(\tau_6+\tau_7+\tau_8+\tau_9)L_3
-(\tau_{10}+\tau_{11})L_4-\tau_{12}L_5
-\tau_{16}(L_3+L_4)/4-\tau_{17}(L_2+L_3)/4
\end{align*}
to $\Delta_1$, and write the result $\widetilde\Delta_1$ in the
form $\boldsymbol\tau\,\mathbf d^T$. The components of $\mathbf d$
are given below.
\begin{xalignat*}2
d_1&=v_1+v_{12}&
d_4&=v_1+v_4+v_8+v_{11}+2v_{12}+v_{14}+v_{15}\\
d_2&=v_2+v_{11}&
d_5&=v_1+v_8+v_9+v_{11}+2v_{12}+v_{14}+v_{15}\\
d_3&=v_8+v_{11}+v_{12}+v_{14}+v_{15}&
d_6&=v_8+v_9+2v_{11}+2v_{12}+2v_{14}+v_{15}\\
&&d_7&=v_4+v_8+2v_{11}+2v_{12}+2v_{14}+v_{15}\\
d_{12}&=v_{10}+v_{11}+v_{12}+v_{13}+v_{14}&
d_8&=v_4+v_{10}+2v_{11}+2v_{12}+v_{13}+2v_{14}\\
d_{13}&=-v_8+v_{13}&
d_9&=v_9+v_{10}+2v_{11}+2v_{12}+v_{13}+2v_{14}\\
d_{14}&=-v_9+v_{14}&
d_{10}&=v_2+v_9+v_{10}+2v_{11}+v_{12}+v_{13}+v_{14}\\
d_{15}&=-v_{10}+v_{15}&
d_{11}&=v_2+v_4+v_{10}+2v_{11}+v_{12}+v_{13}+v_{14}\\[-26pt]
\end{xalignat*}
\begin{align*}
d_{16}&=\frac{v_2}4+\frac{v_4}2+\frac{v_5}4-\frac{v_6}2+\frac{v_7}2
+\frac{v_8}2-\frac{v_9}2+\frac{v_{10}}4+\frac{3v_{11}}2+\frac{v_{12}}4
+\frac{5v_{14}}4\\[4pt]
d_{17}&=\frac{v_1}4+\frac{v_3}4+\frac{v_4}2+\frac{v_6}2-\frac{v_7}2
+\frac{v_8}4-\frac{v_9}2+\frac{v_{10}}2+\frac{v_{11}}4+\frac{3v_{12}}2
+\frac{5v_{14}}4
\end{align*}
We thus have
\begin{align*}
\widetilde\Delta_1\geq
&\left(\frac{\tau_6+\tau_7+\tau_{16}}2-\tau_{13}\right)v_8
-\left(\frac{\tau_{16}+\tau_{17}}2+\tau_{14}\right)v_9
+\left(\frac{\tau_8+\tau_9+\tau_{17}}2-\tau_{15}\right)v_{10}\\[2pt]
&-\frac{\tau_{16}-\tau_{17}}2\left(v_6-v_7
+\frac{v_3-v_5+v_8-v_{10}}2\right).
\end{align*}

Last, $\Delta_2=\boldsymbol\tau\,W\,E\,W^T\,\boldsymbol\tau^T$ is
the quadratic form given by
\begin{align*}
\Delta_2=\;&\tau_1^2+\tau_2^2+\tau_3^2+\tau_4^2+\tau_5^2+\tau_6^2
+\tau_7^2+\tau_8^2+\tau_9^2+\tau_{10}^2+\tau_{11}^2+\tau_{12}^2
+\tau_{13}^2+\tau_{14}^2+\tau_{15}^2\\
&+\tau_{16}^2+\tau_{17}^2+\tau_1\tau_3+\tau_1\tau_4+\tau_1\tau_5
+\tau_1\tau_6+\tau_1\tau_7+\tau_1\tau_{14}+\tau_1\tau_{15}
+\tau_1\tau_{16}+\tau_2\tau_8\\
&+\tau_2\tau_9+\tau_2\tau_{10}+\tau_2\tau_{11}+\tau_2\tau_{12}
+\tau_2\tau_{13}+\tau_2\tau_{14}+\tau_2\tau_{17}+\tau_3\tau_4
+\tau_3\tau_5+\tau_3\tau_{14}\\
&+\tau_3\tau_{16}+\tau_4\tau_5+\tau_4\tau_7-\tau_4\tau_9
-\tau_4\tau_{10}+\tau_4\tau_{15}+\tau_5\tau_6-\tau_5\tau_8
-\tau_5\tau_{11}+\tau_5\tau_{14}\\
&+\tau_5\tau_{15}+\tau_6\tau_7-\tau_6\tau_8-\tau_6\tau_{11}
-\tau_6\tau_{13}+\tau_6\tau_{14}+\tau_6\tau_{15}+\tau_6\tau_{16}
-\tau_7\tau_9-\tau_7\tau_{10}\\
&-\tau_7\tau_{13}+\tau_7\tau_{15}+\tau_7\tau_{16}+\tau_8\tau_9
+\tau_8\tau_{11}+\tau_8\tau_{13}-\tau_8\tau_{15}+\tau_8\tau_{17}
+\tau_9\tau_{10}+\tau_9\tau_{13}\\
&+\tau_9\tau_{14}-\tau_9\tau_{15}+\tau_9\tau_{17}+\tau_{10}\tau_{11}
+\tau_{10}\tau_{12}+\tau_{10}\tau_{13}+\tau_{10}\tau_{14}
+\tau_{11}\tau_{12}+\tau_{11}\tau_{13}\\
&+\tau_{12}\tau_{14}+\tau_{12}\tau_{17}-\tau_{13}\tau_{16}
+\tau_{14}\tau_{16}+\tau_{14}\tau_{17}-\tau_{15}\tau_{17},
\end{align*}
and therefore
\begin{align*}
\Delta_2\geq\;&\tau_1^2+\tau_2^2+\tau_3^2+\tau_4^2+\tau_5^2+\tau_6^2
+\tau_7^2+\tau_8^2+\tau_9^2+\tau_{10}^2+\tau_{11}^2+\tau_{12}^2
+\tau_{13}^2+\tau_{14}^2+\tau_{15}^2\\
&+\tau_{16}^2+\tau_{17}^2
-\tau_{13}(\tau_6+\tau_7+\tau_{16})
-\tau_{15}(\tau_8+\tau_9+\tau_{17})
-(\tau_4+\tau_7)(\tau_9+\tau_{10})\\
&-(\tau_5+\tau_6)(\tau_8+\tau_{11})
+\tau_{16}\left(\tau_{14}+\frac{\tau_6+\tau_7}2\right)
+\tau_{17}\left(\tau_{14}+\frac{\tau_8+\tau_9}2\right)\\
&+(\tau_4\tau_5+\tau_5\tau_6+\tau_6\tau_7+\tau_4\tau_7)
+(\tau_8\tau_9+\tau_9\tau_{10}+\tau_{10}\tau_{11}+\tau_8\tau_{11}).
\end{align*}

Modulo the relations $L_i$, $\Delta$ is congruent to
$\widetilde\Delta=\widetilde\Delta_0+\widetilde\Delta_1+\Delta_2$,
and we have
\begin{align*}
\widetilde\Delta\geq\;&
\left(\frac{v_{12}-v_1}2\right)^2
+\left(\frac{v_{11}-v_2}2\right)^2
+\frac12\left(v_6-v_7+\frac{v_3-v_5+v_8-v_{10}}2
-\frac{\tau_{16}-\tau_{17}}2\right)^2\\
&+\frac12\left(\frac{v_3-v_8}2-\frac{v_5-v_{10}}2\right)^2
+\left(\frac{v_3-v_8}2-\frac{v_4-v_9}2+\frac{v_5-v_{10}}2\right)^2\\
&+\frac{v_3^2+v_4^2+v_5^2}4+\frac{v_3v_8+v_4v_9+v_5v_{10}}2
+\left(\frac{v_8}2+\frac{\tau_6+\tau_7+\tau_{16}}2-\tau_{13}\right)^2\\
&+\left(\frac{v_9}2-\frac{\tau_{16}+\tau_{17}}2-\tau_{14}\right)^2
+\left(\frac{v_{10}}2+\frac{\tau_8+\tau_9+\tau_{17}}2-\tau_{15}\right)^2\\
&+\frac{(\tau_{16}-\tau_{17})^2}8+\frac{\tau_{16}^2+\tau_{17}^2}4
+\tau_1^2+\tau_2^2+\tau_3^2+\tau_{12}^2
+\left(\frac{\tau_6+\tau_7}2\right)^2
+\left(\frac{\tau_8+\tau_9}2\right)^2\\
&+\frac12\bigl((\tau_4+\tau_5)^2+(\tau_{10}+\tau_{11})^2
+(\tau_4+\tau_7-\tau_9-\tau_{10})^2
+(\tau_5+\tau_6-\tau_8-\tau_{11})^2\bigr)\\
\geq\;&0.
\end{align*}

This concludes the proof of equation~\eqref{eq:SemiSmallness}.

To finish the proof of Proposition~\ref{pr:SemiSmallness}, it
remains to study the case of equality. From the minoration above,
one easily sees that $\widetilde\Delta=0$ is possible only if
all the $\tau_k$ vanish, except perhaps $\tau_{13}$, $\tau_{14}$
and $\tau_{15}$, and if
\begin{gather*}
v_1=v_{12},\quad v_2=v_{11},\quad v_3=v_4=v_5=0,\quad v_6=v_7,\\[4pt]
v_8=v_9/2=v_{10},\quad v_8=2\tau_{13},\quad v_9=2\tau_{14},\quad
v_{10}=2\tau_{15}.
\end{gather*}
Substituting into $L_i=0$, we then get
$$v_1=v_2=v_{11}=v_{12},\quad
v_8=v_9=v_{10}=v_{13}=v_{14}=v_{15}=0,\quad
p=v_1+v_6,\quad\boldsymbol\tau=0.$$
Therefore $\mathbf v=(p-2v_6)\mathbf y+v_6\mathbf z$ and
$\mathbf u=\mathbf v+p\mathbf y=2(p-v_6)\mathbf y+v_6\mathbf z$,
as desired.

The proof of the converse (that is, if $\mathbf u$ and $\mathbf v$
have the required form, then
$\widetilde{\mathcal F}^{\mathbf u,\mathbf v}\neq\varnothing$
and \eqref{eq:SemiSmallness} in an equality) is easy and left to
the reader.

\section{The canonical basis and the quantum Frobenius morphism}
\label{se:CanFrob}
In this section, we study the compatibility of the quantum Frobenius
morphism and its splitting with bases of canonical type. These
morphisms require the use of the quantum group, so from now on,
$\mathbf f$ denotes the quantum deformation of the algebra defined
in Section~\ref{ss:BasesOCT}. In addition, we also need a basis that
lifts to the quantum group and can be specialized to a quantum root
of unity. This invites us to restrict our attention to the canonical
basis.

\subsection{Background on the quantum Frobenius morphism}
We follow the notation set up in Lusztig's book~\cite{Lusztig93},
and in particular assume that the conditions (a) and (b) in
Section~35.1.2 of that book hold true.

Let $(d_i)$ be a family of positive integers such that the
matrix $(d_ia_{i,j})$ is symmetric. Let $v$ be an indeterminate.
For $n\in\mathbb N$ and $i\in I$, we define the Gaussian number
$[n]_i$ and the Gaussian factorial $[n]_i!$ as in
Section~\ref{ss:LusztigData}. From these data, we define $\mathbf f$
as the $\mathbb Q(v)$-algebra generated by elements $\theta_i$,
for $i\in I$, submitted to the relations~\eqref{eq:DefAlgF}, in
which the ordinary factorials $p!$ and $q!$ are replaced by their
Gaussian counterparts $[p]_i!$ and $[q]_i!$.

We set $\mathcal A=\mathbb Z[v,v^{-1}]$. It is known that the
$\mathcal A$-subalgebra of $\mathbf f$ generated by the divided
powers $\theta_i^{(n)}=\theta_i^n/[n]_i!$ is an $\mathcal A$-form
${}_{\mathcal A}\mathbf f$ in $\mathbf f$. We can then specialize
the parameter $v$ to any invertible element in a commutative ring
$R$ by a base change $R\otimes_{\mathcal A}{}_{\mathcal A}\mathbf f$
(\cite{Lusztig93}, Section~31.1). Thus for instance, the algebra
$\mathbf f$ from Section~\ref{ss:BasesOCT} is the specialization over
$\mathbb Q$ at the value $v=1$. The canonical basis is in fact an
$\mathcal A$-basis of ${}_{\mathcal A}\mathbf f$, so it induces an
$R$-basis in each specialization ${}_R\mathbf f$.

Let $\ell$ be a positive integer. For $i\in I$, let $\ell_i$ be the
smallest positive integer such that $\ell_id_i\in\ell\mathbb Z$.
We define a new symmetrizable Cartan matrix $A^*=(a_{i,j}^*)$ by
$a_{i,j}^*=a_{i,j}\ell_j/\ell_i$; setting $d_i^*=d_i\ell_i^2$,
the matrix $(d_i^*a_{i,j}^*)$ is symmetric. We then have a
$\mathbb Q(v)$-algebra $\mathbf f^*$ and an $\mathcal A$-form
${}_{\mathcal A}\mathbf f^*$. The simple roots and coroots for the
starred Cartan datum are chosen to be
$\alpha_i^*=\ell_i\alpha_i$ and $(\alpha_i^*)^\vee=\alpha_i^\vee/\ell_i$.

Let $\Phi_{2\ell}$ be the $2\ell$-th cyclotomic polynomial,
and let $R=\mathbb Q[\zeta]/(\Phi_{2\ell}(\zeta))$. Let
${}_R\mathbf f$ and ${}_R\mathbf f^*$ be the specializations
of $\mathbf f$ and $\mathbf f^*$ over $R$ at the value $v=\zeta$.
There is then an algebra homomorphism
$Fr_\ell:{}_R\mathbf f\to{}_R\mathbf f^*$ that maps the generator
$\theta_i^{(n)}$ to $\theta_i^{(n/\ell_i)}$ if $n$ is a multiple of
$\ell_i$ and to $0$ otherwise; this morphism is called the quantum
Frobenius map. In the other direction, there is an algebra
homomorphism $Fr'_\ell:{}_R\mathbf f^*\to{}_R\mathbf f$ that maps
$\theta_i^{(n)}$ to $\theta_i^{(n\ell_i)}$; this map is called the
quantum Frobenius splitting.

Let $B(-\infty)^*$ be the analogue of the crystal $B(-\infty)$ for the
Cartan matrix $A^*$ and the algebra $\mathbf f^*$. By construction,
the Frobenius splitting $Fr'_\ell$ has some kind of compatibility
with the conditions \ref{it:BOCTa}--\ref{it:BOCTc} in the definition
of a basis of canonical type. One may thus expect the existence of a
map $S_\ell:B(-\infty)^*\to B(-\infty)$ that reflects the action of
$Fr'_\ell$ at the level of the crystals. Such a map $S_\ell$ has been
constructed by Kashiwara (\cite{Kashiwara96}, Theorems~3.2 and~5.1);
it satisfies
\begin{gather*}
\wt(S_\ell(b))=\wt(b),\qquad
\varepsilon_i(S_\ell(b))=\ell_i\varepsilon_i(b),\qquad
\varphi_i(S_\ell(b))=\ell_i\varphi_i(b),\\[4pt]
S_\ell(\tilde e_ib)=\tilde e_i^{\ell_i}S_\ell(b),\qquad
S_\ell\bigl(\tilde f_ib\bigr)=\tilde f_i^{\ell_i}S_\ell(b).
\end{gather*}
These equations take into account the convention that the
simple roots and coroots for the starred root datum are
given by $\alpha_i^*=\ell_i\alpha_i$ and
$(\alpha_i^*)^\vee=\alpha_i^\vee/\ell_i$.

\subsection{Compatibility up to a filtration}
\label{ss:CompUptoFilt}
As promised, we now study the compatibility of $Fr_\ell$ and $Fr'_\ell$
with the canonical bases of ${}_R\mathbf f$ and ${}_R\mathbf f^*$.
The best compatibility we could hope for would be
\begin{equation}
\label{eq:CompFrobCan}
Fr_\ell(G(b'))=
\begin{cases}
G(b'')&\text{if $b'=S_\ell(b'')$}\\
0&\text{if $b'\notin\im S_\ell$}
\end{cases}
\qquad\text{and}\qquad
Fr'_\ell(G(b''))=G(S_\ell(b'')),
\end{equation}
where $G(b')$ and $G(b'')$ denote the elements in the canonical
bases of ${}_R\mathbf f$ and ${}_R\mathbf f^*$ that correspond to
$b'\in B(-\infty)$ and $b''\in B(-\infty)^*$.

This property holds true in types $A_1$, $A_2$, $A_3$ and $B_2$.
I owe this nice observation to Littelmann, who checked it using
the explicit formulas for the canonical basis given by Lusztig
(\cite{Lusztig90}, Section~3.4) and Xi~\cite{Xi99a,Xi99b}. Alas,
\eqref{eq:CompFrobCan} fails in types $A_5$ and~$D_4$, as we will see
in Section~\ref{ss:FrobCexA5D4}.

One can however hope to restore the compatibility by working with
filtrations, so as to be able to neglect undesired terms in
the expansion of $Fr_\ell(G(b'))$ and $Fr'_\ell(G(b''))$. To this
aim, we must study how these terms compare with the expected one.

The next proposition is a crude result, whose proof relies solely
on the property that the canonical basis is of canonical type.

\begin{proposition}
\label{pr:FrobCanStr}
Let $\ell\geq1$ and let $(b',b'')\in(B(-\infty)^*)^2$.
\begin{enumerate}
\item
\label{it:FCSa}
In order that $G(b'')$ actually occurs in the expansion of
$Fr_\ell(G(S_\ell(b')))$ on the canonical basis of ${}_R\mathbf f^*$,
it is necessary that $b'\leq_{\str}b''$. Moreover, $G(b')$ occurs
with coefficient $1$ in the expansion of $Fr_\ell(G(S_\ell(b')))$.
\item
\label{it:FCSb}
In order that $G(S_\ell(b'))$ actually occurs in the expansion of
$Fr'_\ell(G(b''))$ on the canonical basis of ${}_R\mathbf f$,
it is necessary that $b''\leq_{\str}b'$. Moreover, $G(S_\ell(b''))$
occurs with coefficient $1$ in the expansion of $Fr_\ell(G(b''))$.
\end{enumerate}
\end{proposition}
\begin{proof}
To avoid confusion with the star in the notation $\mathbf f^*$,
it will be convenient to denote duality with a superscript $\vee$.
Thus $({}_R\mathbf f)^\vee=\Hom_R({}_R\mathbf f,R)$ and
$({}_R\mathbf f^*)^\vee=\Hom_R({}_R\mathbf f^*,R)$. The elements
in the dual canonical bases of these algebras are denoted by
$G(b)^\vee$, where $b$ is in $B(-\infty)$ or in $B(-\infty)^*$,
respectively.

Let $\ell\geq1$ and let $(b,b'')\in B(-\infty)\times B(-\infty)^*$.
Choose $i\in I$ and set $k=\lfloor\varphi_i(b)/\ell_i\rfloor$, the
largest integer smaller than or equal to $\varphi_i(b)/\ell_i$.
By Theorem~14.3.2 in~\cite{Lusztig93}, there are elements
$x_n\in{}_R\mathbf f$ such that
$$\theta_i^{(\varphi_i(b))}\,G\bigl(\tilde f_i^{\max}b\bigr)=
G(b)+\sum_{n>\varphi_i(b)}\theta_i^{(n)}x_n.$$
Applying $Fr_\ell$ to this equation, we obtain that modulo
$\theta_i^{k+1}{}_R\mathbf f^*$,
$$Fr_\ell(G(b))\equiv
\begin{cases}
\theta_i^{(k)}Fr_\ell\bigl(G\bigl(\tilde f_i^{\max}b\bigr)\bigr)
&\text{if $\varphi_i(b)=k\ell_i$,}\\
0&\text{otherwise.}
\end{cases}$$
Thus, if $G(b'')$ actually occurs in the expansion of
$Fr_\ell(G(b))$ on the canonical basis of ${}_R\mathbf f^*$,
then either $\varphi_i(b'')\geq k+1$, or $\varphi_i(b)=k\ell_i$
and $\varphi_i(b'')=k$; in any case,
$\varphi_i(b'')\geq\varphi_i(b)/\ell_i$. Moreover,
when $\varphi_i(b'')=\varphi_i(b)/\ell_i$,
$$\bigl\langle G(b'')^\vee,Fr_\ell(G(b))\bigr\rangle=
\bigl\langle G(b'')^\vee,\theta_i^{(k)}Fr_\ell\bigl(G\bigl(\tilde
f_i^{\max}b\bigr)\bigr)\bigr\rangle=\bigl\langle G\bigl(\tilde
f_i^{\max}b''\bigr)^\vee,Fr_\ell\bigl(G\bigl(\tilde
f_i^{\max}b\bigr)\bigr)\bigr\rangle.$$

Now let $b'\in B(-\infty)^*$. Applying the previous reasoning to
$b=S_\ell(b')$ and using that $\sigma$ commutes with $Fr_\ell$ and
with $S_\ell$, we eventually obtain:
\begin{itemize}
\item
If $\bigl\langle G(b'')^\vee,Fr_\ell(G(S_\ell(b')))\bigr\rangle\neq0$,
then $\varphi_i(b'')\geq\varphi_i(b')$ for each $i\in I$.
\item
If $(c',c'')\in(B(-\infty)^*)^2$ is such that $(b',b'')\approx(c',c'')$,
then
$$\bigl\langle G(b'')^\vee,Fr_\ell(G(S_\ell(b')))\bigr\rangle=
\bigl\langle G(c'')^\vee,Fr_\ell(G(S_\ell(c')))\bigr\rangle.$$
\end{itemize}
This proves~\ref{it:FCSa}.

The proof of~\ref{it:FCSb} is completely analogous.
\end{proof}

Suppose now that the Cartan matrix $A$ is of finite type and adopt
the notation of Section~\ref{ss:LusztigData}. Observing that the
Weyl group for the Cartan matrices $A$ and $A^*$ are the same, we
can use the same $\mathbf i\in\mathscr X$ to construct a PBW basis
in $\mathbf f$ and in $\mathbf f^*$. Abusing slightly the notation,
we denote the elements in these bases by
$E_{\mathbf i}^{(\mathbf n)}$ and ${}^*E_{\mathbf i}^{(\mathbf n)}$,
for $\mathbf n\in\mathbb N^N$ --- the abuse is that in
Section~\ref{ss:LusztigData}, $E_{\mathbf i}^{(\mathbf n)}$ was
a basis of $U_q(\mathfrak n_+)$, which we now transport to $\mathbf f$.
By Section~7.1 in~\cite{Lusztig92}, these bases are in fact bases of
${}_{\mathcal A}\mathbf f$ and ${}_{\mathcal A}\mathbf f^*$, so they
can be specialized to ${}_R\mathbf f$ and ${}_R\mathbf f^*$. Lastly,
we define the map $S_{\mathbf i,\ell}:(n_1,\ldots,n_N)\mapsto
(\ell_{i_1}n_1,\ldots,\ell_{i_N}n_N)$ from $\mathbb N^N$ to
itself, and we denote by $\leq_{\mathbf i}$ and $\leq_{\mathbf i}^*$
the orders on $\mathbb N^N$ relative to the Cartan matrices $A$
and $A^*$, respectively.

\begin{lemma}
\label{le:FrobPBW}
Let $\ell\geq1$ and let $\mathbf i\in\mathscr X$.
\begin{enumerate}
\item
\label{it:FPBWa}
The diagram
$$\xymatrix@C=40pt{B(-\infty)^*\ar[r]^{S_\ell}
\ar[d]_{n_{\mathbf i}}&B(-\infty)\ar[d]^{n_{\mathbf i}}\\
\mathbb N^N\ar[r]^{S_{\mathbf i,\ell}}&\mathbb N^N}$$
commutes.
\item
\label{it:FPBWb}
Let $\mathbf m$ and $\mathbf n$ in $\mathbb N^N$. Then
$\mathbf n\leq_{\mathbf i}^*\mathbf m$ if and only if
$S_{\mathbf i,\ell}(\mathbf n)\leq_{\mathbf i}
S_{\mathbf i,\ell}(\mathbf m)$.
\item
\label{it:FPBWc}
Let $\mathbf n\in\mathbb N^N$. If $\mathbf n$
does not belong to the image of $S_{\mathbf i,\ell}$,
then $Fr_\ell\bigl(E_{\mathbf i}^{(\mathbf n)}\bigr)=0$.
Otherwise, $Fr_\ell\bigl(E_{\mathbf i}^{(\mathbf n)}\bigr)=
{}^*E_{\mathbf i}^{(\mathbf m)}$, where
$\mathbf m=S_{\mathbf i,\ell}^{-1}(\mathbf n)$.
\item
\label{it:FPBWd}
Let $\mathbf m$ and $\mathbf n$ in $\mathbb N^N$. If
$E_{\mathbf i}^{(\mathbf m)}$ actually occurs in the expansion of
$Fr'_\ell\bigl({}^*E_{\mathbf i}^{(\mathbf n)}\bigr)$ on the PBW
basis of ${}_R\mathbf f$, then $\mathbf m\geq_{\mathbf i}S_{\mathbf
i,\ell}(\mathbf n)$.
\end{enumerate}
\end{lemma}
\begin{proof}
Given $(\mathbf i,\mathbf j)\in\mathscr X^2$, let
$R_{\mathbf i}^{\mathbf j}$ be the composition
$\mathbb N^N\xrightarrow{\mathbf b_{\mathbf i}}B(-\infty)
\xrightarrow{\mathbf n_{\mathbf j}}\mathbb N^N$ and let
${}^*R_{\mathbf i}^{\mathbf j}$ be the composition
$\mathbb N^N\xrightarrow{\mathbf b_{\mathbf i}}B(-\infty)^*
\xrightarrow{\mathbf n_{\mathbf j}}\mathbb N^N$.
Using the explicit formulas for these piecewise linear bijections
$R_{\mathbf i,\mathbf j}$ and ${}^*R_{\mathbf i,\mathbf j}$ (see
Section~12.6 in \cite{Lusztig92} and Theorem~5.2 and Proposition~7.1
in~\cite{BerensteinZelevinsky01}), one checks that the diagram
$$\xymatrix@C=60pt{\mathbb N^N\ar[r]^{S_{\mathbf i,\ell}}
\ar[d]_{{}^*R_{\mathbf i}^{\mathbf j}}&\mathbb N^N
\ar[d]^{R_{\mathbf i}^{\mathbf j}}\\
\mathbb N^N\ar[r]^{S_{\mathbf j,\ell}}&\mathbb N^N}$$
commutes. From there, one shows assertion~\ref{it:FPBWa} by
induction on the weight, using the same arguments as those used
in \cite{BerensteinZelevinsky01}, proof of Theorem~5.7 (see in
particular p.~112, l.~5--12).

Assertion~\ref{it:FPBWb} follows from the definitions by a
straightforward computation.

Assertion~\ref{it:FPBWc} comes from the fact that the quantum
Frobenius morphism $Fr_\ell$ is compatible with Lusztig
symmetries $T'_{i,-1}$ (\cite{Lusztig93}, Section~41.1.9), which
are the main ingredient in the construction of the PBW bases.

To prove~\ref{it:FPBWd}, one begins with the particular case
where all entries of $\mathbf n$ but one vanish. Certainly
then $S_\ell(\mathbf n)$ has the same property, so every
$\mathbf m\in\mathbb N^N$ such that
$|\mathbf m|=|S_{\mathbf i,\ell}(\mathbf n)|$ satisfies
$\mathbf m\geq_{\mathbf i}S_{\mathbf i,\ell}(\mathbf n)$.
This obvioulsy implies the desired property. The general case
then follows by induction on the number of nonzero entries in
$\mathbf n$, using Lemma~\ref{le:ConvPBW} and the fact that
$Fr'_\ell$ is a morphism of algebras.
\end{proof}

The next proposition states that the compatibility condition
\eqref{eq:CompFrobCan} can be restored by filtering $\mathbf f$
and $\mathbf f^*$ with the help of $\leq_{\pol}$. It thus tells
us that $S_\ell$ is the crystal version of $Fr'_\ell$.

\begin{proposition}
\label{pr:FrobCanPol}
Let $\ell\geq1$ and let $(b',b'')\in B(-\infty)\times B(-\infty)^*$.
\begin{enumerate}
\item
\label{it:FCPa}
If $G(b'')$ actually occurs in the expansion of $Fr_\ell(G(b'))$
on the canonical basis of ${}_R\mathbf f^*$, then
$b'\leq_{\pol}S_\ell(b'')$.
\item
\label{it:FCPb}
If $G(b')$ actually occurs in the expansion of $Fr'_\ell(G(b''))$
on the canonical basis of ${}_R\mathbf f$, then $S(b'')\leq_{\pol}b'$.
\end{enumerate}
\end{proposition}
\begin{proof}
This follows from Proposition~\ref{pr:PolOrderPBWOrder},
Corollary~\ref{co:TriangPBWCan}, and Lemma~\ref{le:FrobPBW}
by routine arguments.
\end{proof}

\begin{other}{Remark}
\label{rk:FrobHall}
In the Hall algebra model for quantum groups, the natural basis of
the Hall algebra corresponds to a PBW basis (see~\cite{Ringel96} for
a survey). The compatibility of the quantum Frobenius morphism with
the PBW bases (Lemma~\ref{le:FrobPBW}~\ref{it:FPBWc}) then leads to
an interpretation of $Fr_\ell$ within the framework of Hall
algebras, which can be used as an alternate definition of $Fr_\ell$
\cite{McGerty10}. Conversely, one can adopt this Hall algebra approach
to show the compatibility of $Fr_\ell$ with the automorphisms
$T'_{i,-1}$, using Theorem~6 in~\cite{Ringel96} or Theorem~13.1
in~\cite{SevenhantVandenBergh99}.
\end{other}

\subsection{Counterexamples in type $A_5$ and $D_4$}
\label{ss:FrobCexA5D4}
As mentioned at the beginning of Section~\ref{ss:CompUptoFilt},
counterexamples to \eqref{eq:CompFrobCan} do exist in type $A_5$
and~$D_4$.

Let us take $d_i=1$ for each $i$, whence $\ell_i=\ell$ and $A^*=A$,
and let us adopt the notation of Section~\ref{ss:StatRes}. Since
\begin{xalignat*}2
Fr_\ell(\xi_p)&=
\begin{cases}
\xi_{p/\ell}&\text{if $\ell$ divides $p$,}\\
0&\text{otherwise,}
\end{cases}
&Fr_\ell(\eta_p)&=
\begin{cases}
\eta_{p/\ell}&\text{if $\ell$ divides $p$,}\\
0&\text{otherwise,}
\end{cases}\\[8pt]
Fr'_\ell(\xi_p)&=\xi_{\ell p},
&Fr'_\ell(\eta_p)&=\eta_{\ell p},
\end{xalignat*}
Proposition~\ref{pr:SmallA5D4} and
Theorem~\ref{th:MainA5D4}~\ref{it:MADb} lead to
\begin{align*}
&Fr_\ell(G(b_{p,0}))=
\begin{cases}
G(b_{p/\ell,0})&\text{if $\ell$ divides $p$,}\\
0&\text{otherwise,}
\end{cases}\\[8pt]
&Fr_2(G(b_{0,1}))+G(b_{1,0})=0,\\[4pt]
&Fr_2(G(b_{0,2}))+Fr_2(G(b_{2,1}))=G(b_{0,1}),\\[4pt]
&Fr_4(G(b_{0,2}))+Fr_4(G(b_{2,1}))+G(b_{1,0})=0,
\end{align*}
and to
\begin{align*}
&Fr'_\ell(G(b_{p,0}))=G(b_{\ell p,0}),\\[4pt]
&Fr'_\ell(G(b_{0,1}))=G(b_{0,\ell})+G(b_{2,\ell-1})
+G(b_{4,\ell-2})+\cdots+G(b_{2\ell-2,1}).
\end{align*}

In addition, in the case (III), calculations made by a computer
running GAP and its package QuaGroup \cite{GAP4,DeGraaf07} lead to
further relations in ${}_{\mathcal A}\mathbf f$. Since all $d_i=1$,
we may drop the subscript~$i$ in the notation for the
Gaussian numbers. Let us introduce a linear operator
$R_{i,\ell}:{}_{\mathcal A}\mathbf f\to{}_{\mathcal A}\mathbf f$~by
$$R_{i,\ell}(x)=\theta_i^{(\ell-1)}x\theta_i-[\ell-2]\,
\theta_i^{(\ell)}x,$$
where $i\in\{1,3,4\}$ and $\ell\in\mathbb N$. Then
\begin{align*}
&\theta_2^{(2)}\Bigl(R_{1,3}\circ R_{3,3}\circ R_{4,3}
\bigl(\theta_2^{(2)}\bigr)\Bigr)\,
\theta_2^{(2)}=G(b_{1,1})+[2]^2\,G(b_{3,0}),\\[4pt]
&\theta_2^{(3)}\Bigl(R_{1,4}\circ R_{3,4}\circ R_{4,4}
\bigl(\theta_2^{(2)}\bigr)\Bigr)\,
\theta_2^{(3)}=G(b_{2,1})+[3]^2\,G(b_{4,0}).
\end{align*}

Using the congruences $[3]^2\equiv1$ modulo $\Phi_4$ or $\Phi_8$
and $[2]^2\equiv1\pmod{\Phi_6}$, we deduce
\begin{align*}
&Fr_2(G(b_{2,1}))+G(b_{2,0})=0,\\[4pt]
&Fr_3(G(b_{1,1}))+G(b_{1,0})=0,\\[4pt]
&Fr_4(G(b_{2,1}))+G(b_{1,0})=0,
\end{align*}
whence
$$Fr_2(G(b_{0,2}))=G(b_{0,1})+G(b_{2,0}).$$

Pierre Baumann\\[4pt]
Institut de Recherche Math\'ematique Avanc\'ee,
Universit\'e de Strasbourg et CNRS,
7 rue Ren\'e Descartes,
67084 Strasbourg Cedex,
France\\[4pt]
\href{mailto:p.baumann@unistra.fr}{p.baumann@unistra.fr}
\medskip

The author acknowledges the support of the ANR,
project~ANR-09-JCJC-0102-01.
\end{document}